\documentclass{article}
\usepackage{graphicx} 
\usepackage{amssymb,latexsym,amsfonts,amsthm,a4wide}              
\usepackage{mathrsfs}
\usepackage{mathtools}
\usepackage{bm}%
\usepackage{subfigure}
\usepackage{float}
\usepackage{cases}
\usepackage{textcomp} %
\usepackage{hyperref}
\usepackage{booktabs}%
\usepackage{color}
\usepackage{algorithm, algpseudocode} 
\newtheorem{theorem}{Theorem}[section]
\newtheorem{lemma}{Lemma}[section]
\newtheorem{remark}{Remark}

\numberwithin{equation}{section}

\newcommand{\normmm}[1]{{\left\vert\kern-0.25ex\left\vert\kern-0.25ex\left\vert #1   
   \right\vert\kern-0.25ex\right\vert\kern-0.25ex\right\vert}}                       

\title{Some semi-decoupled algorithms with optimal convergence for a four-field linear thermo-poroelastic model}

 \author{Li Ziliang \thanks {Harbin Engineering University}
       \and Cai Mingchao \thanks{Morgan State University}  
       \and Li Jingzhi \thanks {Southern University of Science and Technology }
        \and
       Liu Qiang\thanks{Shenzhen University}\footnotemark[3]
        }
\date{}
\begin{document}

\maketitle

\begin{abstract}
We propose three semi-decoupled algorithms for efficiently solving a four-field thermo-poroelastic model. The first two algorithms adopt a sequential strategy: at the initial time step, all variables are computed simultaneously using a monolithic solver; thereafter, the system is split into a mixed linear elasticity subproblem and a coupled pressure–temperature reaction-diffusion subproblem. The two variants differ in the order in which these subproblems are solved. To further improve computational efficiency, we introduce a parallel semi-decoupled algorithm. In this approach, the four-field system is solved monolithically only at the first time step, and the two subproblems are then solved in parallel at subsequent time levels.
None of the three algorithms requires iterative procedures at each time step, and are free from stabilization. Rigorous analysis confirms their unconditional stability, optimal convergence rates, and robustness under a wide range of physical parameter settings. These theoretical results are further validated by numerical experiments.



\end{abstract}

\maketitle

\section{Introduction}
The theoretical framework of poroelasticity was first established by Terzaghi in his seminal work on one-dimensional consolidation theory~\cite{terzaghi1943theoretical}. Building on this foundation, Biot developed a rigorous mathematical formulation that describes the dynamic interaction between a deformable porous solid matrix and the interstitial fluid under isothermal conditions, leading to what is now widely known as Biot’s theory of poroelasticity~\cite{biot1941general, biot1955theory}. These classical models have since been extended to incorporate thermal effects, giving rise to the thermo-poroelastic model, which integrates temperature-driven processes into the coupled system of equations~\cite{florez2018linear, selvadurai2016thermo}. Due to its generality and strong physical grounding, the thermo-poroelastic framework finds widespread applications in a range of disciplines, including geomechanics (e.g., soil consolidation, geothermal reservoirs), biomedical engineering (e.g., thermal ablation of tumors, drug transport in tissues), and environmental sciences (e.g., carbon dioxide sequestration and thermal remediation).

The thermo-poroelastic model governs the interplay among three fundamental physical phenomena: fluid flow transport, mechanical deformation, and heat conduction. Mathematically, this results in a coupled system of three interdependent partial differential equations: the mass conservation equation for fluid flow, the momentum conservation equation for mechanical response, and the energy conservation equation accounting for thermal diffusion. Depending on the constitutive laws and nonlinearities considered, the system can be either linear or nonlinear. In this work, we restrict our attention to the linear model.  Our focus in this paper is to develop fast, efficient, and robust numerical algorithms for the linear thermo-poroelastic system that can support large-scale simulations in these complex multiphysics settings.

Recent significant progress in the mathematical analysis and numerical approximation of thermo-poroelasticity has enriched both theoretical and computational frameworks. The well-posedness of nonlinear thermo-poroelastic models, including existence, uniqueness, energy estimates, and regularity results, has been established in studies such as~\cite{brun2019well}.
On the computational front, significant progress has been made through both fully coupled and decoupled numerical strategies. Among fully coupled approaches, Galerkin and mixed finite element methods for nonlinear formulations have been explored in~\cite{zhang2024coupling, zhang2022galerkin}, while discontinuous Galerkin techniques for linear and nonlinear problems have been developed in~\cite{antonietti2023discontinuous, bonetti2024robust, bonetti2023numerical}. Efficient solution strategies, including preconditioning techniques for linear thermo-poroelastic systems, are presented in~\cite{cai2025parameter}. A hybrid finite element framework combining mixed, characteristic, and Galerkin elements for different fields has also been proposed~\cite{Zhang2023mfe}.
In the decoupled category, methods are typically divided into sequential and iterative types. Sequential decoupling schemes have been enhanced through four-field formulations that introduce displacement divergence as an auxiliary variable~\cite{chen2022multiphysics}, and further extended to nonlinear models with convective effects~\cite{Ge2023analysis}. Iterative decoupling algorithms, including partially and fully decoupled five-field formulations involving heat and Darcy fluxes, are investigated in~\cite{brun2020monolithic}. More recent developments integrate reduced-order modeling to enhance efficiency~\cite{Ballarin2024projection}, while parameter-robust and stability-insensitive iterative solvers are proposed in~\cite{cai2025efficient}.

Although a variety of numerical methods have been developed for thermo-poroelastic models, the construction of efficient and stable split-parallel algorithms—capable of enabling concurrent computation of subproblems—remains underexplored. The present study advances existing methodologies by developing a class of semi-decoupled algorithms for the thermo-poroelastic system that eliminate the need for iterative coupling between subproblems, generalizing the techniques first proposed for Biot’s model in~\cite{cai2023some, zhao2025optimally}. Specifically, we present two sequentially split semi-decoupled algorithms and one parallel semi-decoupled algorithm.
All three algorithms begin with a fully coupled solve at the initial time step and subsequently decouple the system into two subsystems: a mixed linear elasticity problem and a coupled reaction-diffusion system for fluid pressure and temperature. The proposed parallel algorithm (Algorithm 3 in Section~3) achieves true concurrency by using a two-step time discretization to break the temporal interdependence between subsystems. This is in contrast to the approach in~\cite{cai2023some}, where subproblems are solved sequentially at each time level due to the dependence of temporal derivatives on both current and previous solutions.
In our parallel algorithm, the elasticity subproblem is solved by approximating the pressure and temperature using values from the previous time step, enabling independent computation of displacement and total pressure. Concurrently, the reaction-diffusion subproblem estimates the total pressure derivative using information from the two most recent time levels, thereby allowing the pressure and temperature fields to be updated without requiring the current displacement. A key innovation of our method lies in its ability to avoid the use of stabilization terms, which are typically required in existing split-parallel algorithms to ensure stability~\cite{zhao2025optimally}. By design, the proposed algorithm achieves stability and accuracy without additional stabilization, resulting in reduced computational overhead and enhanced efficiency for large-scale simulations.


The structure of this paper is as follows. Section \ref{sec: PDE model} introduces the thermo-poroelasticity model together with its mathematical reformulation. In Section \ref{sec: algorithms}, we present two sequentially split semi-decoupled algorithms and one parallel semi-decoupled algorithm. Section \ref{sec: convergent} is devoted to establishing the optimal convergence of the parallel algorithm and demonstrating its robustness under a broad range of parameter settings. Section \ref{sec: experiments} reports numerical experiments that validate the theoretical results and highlight the efficiency of the proposed methods. Finally, conclusions are drawn in Section \ref{sec: conclusions}

\section{Mathematical model and its reformulation}\label{sec: PDE model}

We introduce some notations that will be utilized throughout this paper.
Let $\Omega\in\mathbb R^d$, $d=2$ or $3$ be a bounded domain and $J=(0,\tau)$ with $\tau$ being the final time. 
For $1\le s < \infty$, let $L^s(\Omega)=\{v:\int_\Omega|v|^s<\infty \}$ represent the standard Banach space, with the associated norm $\Vert v\Vert_{L^s(\Omega)}=(\int_\Omega |v|^s)^{\tfrac1s}$. We use $(\cdot,\cdot)$ to denote the $L^2$ inner product. 
Let $W^{s,m}(\Omega)$ be the Sobolev space of functions in $L^m(\Omega)$, admitting weak derivatives up to order $s$ in the same space.
In particular, define $H^m(\Omega):= W^{2,m}(\Omega)$ as the special Sobolev space with $p=2$ and define $H_0^m(\Omega)=\{v\in H^1(\Omega):v|_{\partial \Omega}=0 \}$.
 For a Banach space $X$, we give $L^s(J;X)=\{v: (\int_J\|v\|_{X}^s)^{\tfrac1s}<\infty\}$.
For $v \in L^s(J;X)$, define the associated norm
$\Vert v\Vert_{ L^s(J;X)}= (\int_J\|v\|_X^s)^{\tfrac1s}$.
For $v\in L^2(J;L^2(\Omega))$, it is a square-integrable function in space $L^2(\Omega)$ in time.  
Moreover, 
we also use $v_t$ to denote the partial derivative of the function $ v$ in $H^1(\Omega)$ with respect to $t$.

\subsection{A linear thermo-poroelasticity model}
We consider the following linear thermo-poroelasticity model, as described in \cite{brun2018upscaling, cai2025efficient, cai2025parameter}, to determine the displacement field $\bm{u}$, fluid pressure $p$, and temperature $T$ with respect to a reference value. The governing equations are given by
\begin{equation}\label{TP_model}
\begin{aligned}
-\nabla \cdot (2\mu \bm{\varepsilon}(\bm{u}) + \lambda \nabla \cdot \bm{u} \bm{I}) + \alpha \nabla p + \beta \nabla T &= \bm{f}, && \text{in } \Omega \times J, \\
\tfrac{\partial}{\partial t}(c_0 p - b_0 T + \alpha \nabla \cdot \bm{u}) - \nabla \cdot (\bm{K} \nabla p) &= g, && \text{in } \Omega \times J, \\
\tfrac{\partial}{\partial t}(a_0 T - b_0 p + \beta \nabla \cdot \bm{u}) - \nabla \cdot (\bm{\Theta} \nabla T) &=  H_s, && \text{in } \Omega \times J.
\end{aligned}
\end{equation}
The operator $\bm\varepsilon$ is defined as $\bm\varepsilon(\bm u)=\tfrac{1}{2}(\nabla\bm u+\nabla\bm u^T)$ and
the matrix $\bm I$  is the identity tensor. The right-hand side functions $\bm f$, $g$, and $H_{s}$ denote the body force, the mass source, and the heat source, respectively. $\alpha$ is the Biot-Willis constant and $\beta$ is the thermal stress coefficient. $a_0,b_0$, and $c_0$ represent the effective thermal capacity, the thermal dilation coefficient, and the specific storage coefficient, respectively. The matrix parameter $\bm K = (K_{ij})^d_{ij=1}$ 
denotes the permeability divided by fluid viscosity and matrix parameter  $\bm\Theta = (\Theta_{ij})^d_{ij=1}$
denotes the effective thermal conductivity. $\mu$ and $\lambda$ represent the Lamé parameters, which can be formulated in terms of Young's modulus $E$ and Poisson's ratio $\nu$:
\[
\lambda=  \tfrac{E\nu}{(1+\nu)(1-2\nu)},\qquad \mu=\tfrac{E}{2(1+\nu)}.
\]
The model in this paper is expressed in terms of dimensionless variables and coefficients.

To ensure the well-posedness of the problem, appropriate boundary and initial conditions must be imposed. In this work, we consider mixed boundary conditions for $\bm{u}$, with Dirichlet conditions on part of the boundary and Neumann conditions on the remainder. Specifically, we assume $\partial \Omega = \Gamma_d \cup \Gamma_n$ with $\Gamma_d \neq \partial \Omega$, where $\Gamma_d$ and $\Gamma_n$ denote the Dirichlet and Neumann boundaries for $\bm{u}$, respectively. For the pressure $p$ and temperature $T$, Dirichlet boundary conditions are prescribed. Without loss of generality, we assume all boundary conditions are homogeneous.
Specifically, the boundary conditions are given by
\begin{equation}\label{eq: B_C}
\begin{aligned}
    \bm{u} = 0, \quad& \text{on } \Gamma_d,\\
     ( 2\mu \bm{\varepsilon}(\bm{u}) + \lambda \nabla \cdot \bm{u} \bm{I} - \alpha  p\bm I - \beta  T \bm I)\bm n = 0 , \quad& \text{on } \Gamma_n   ,    \\
    p=0, \quad& \text{on } \partial\Omega,  \\
    T=0, \quad& \text{on } \partial\Omega.
\end{aligned}
\end{equation}
Here, $\bm{n}$ is the unit outward normal to the boundary.
The initial conditions are given by 
 \begin{equation}\label{eq: I_C}
\bm{u}(\cdot, 0) = \bm{u}^0, \quad p(\cdot, 0) = p^0, \quad T(\cdot, 0) = T^0.
\end{equation}

We now introduce the four-field formulation for the above model \eqref{TP_model}. More clearly, following the approach presented in literature in \cite{lee2017parameter, oyarzua2016locking, cai2025efficient}, we define an auxiliary variable to represent the volumetric contribution to the total stress:  
$$  
\xi = -\lambda \nabla \cdot \bm{u} + \alpha p + \beta T,
$$  
 which is commonly referred to as the pseudo-total pressure \cite{antonietti2023discontinuous}. Substituting this variable into equation \eqref{TP_model}, the four-field linear thermo-poroelasticity problem is as follows.

\begin{equation}\label{TP_Model_four}
\begin{aligned} 
       -2\mu \nabla\cdot \bm\varepsilon(\bm u) +\nabla\xi 
                      &=\bm f, \\
      -\lambda \nabla\cdot\bm u-\xi+\alpha p+\beta T
                      &=0, \\
-\tfrac\alpha\lambda\tfrac\partial{\partial t}\xi+(c_0+\tfrac{\alpha^2}\lambda)\tfrac\partial{\partial t}p
                 +(\tfrac{\alpha\beta}\lambda-b_0)\tfrac\partial{\partial t}T
                      -\nabla\cdot(\bm K\nabla p)&=g, \\
-\tfrac\beta\lambda\tfrac\partial{\partial t}\xi +(\tfrac{\alpha\beta}\lambda-b_0)\tfrac\partial{\partial t}p
+(a_0+\tfrac{\beta^2}\lambda)\tfrac\partial{\partial t}T
                  -\nabla\cdot(\bm \Theta\nabla T)
                      &= H_{s}. 
\end{aligned}
\end{equation}
We still use the initial condition (\ref{eq: I_C}) and the Dirichlet boundary conditions (\ref{eq: B_C}) for \eqref{TP_Model_four}. In practical situations, nonhomogeneous Dirichlet and Neumann boundary conditions are commonly encountered. The analysis performed for homogeneous boundary conditions can be straightforwardly extended to accommodate these cases.

As outlined in \cite{brun2020monolithic}, we make the following \textbf{Assumptions} for the model parameters throughout this paper (similar assumptions are also discussed in \cite{chen2022multiphysics, zhang2024coupling}):
 
(A1) $\bm{K}$ and $\bm\Theta$ are constant in time, symmetric, and positive definite. There exist $k_m>0$
and $k_M>0$ such that
$$ k_m|\bm v|^2 \le \bm v^T\bm K(x)\bm v,\quad |\bm K(x)\bm v|\le k_M|\bm v|,\quad \forall \bm v\in\mathbb R^d , $$
where $|\cdot|$ denotes the $l^2$-norm of a vector in $\mathbb R^d$.
There exist  $\theta_m>0$ and $\theta_M>0$ such that
$$ \theta_m|\bm v|^2 \le \bm v^T\bm\Theta(x)\bm v,\quad |\bm\Theta(x)\bm v|\le\theta_M|\bm v|,\quad \forall \bm v\in\mathbb R^d . $$

(A2) The constants \(\lambda\), \(\mu\), \(\alpha\) and \(\beta\) are strictly positive constants.

(A3) The constants $a_0,b_0 $ and $c_0$ are such that $a_0,~ c_0 > b_0 \ge 0$ .

In addition to \textbf{Assumptions} (A1)-(A3), throughout the remainder of this paper, we assume that the initial conditions satisfy \(\bm{u}^0 \in [H^1(\Omega)]^d\) and \(p^0, T^0 \in L^2(\Omega)\), while the source terms satisfy \(\bm{f} \in [L^2(\Omega)]^d\), \(g \in L^2(\Omega)\), and \( H_{s} \in L^2(\Omega)\).


Define the following spaces, 
$$
\bm V=[ H^1_{0,\Gamma_d}(\Omega) ]^d,~Q=L^2(\Omega),~W=H_0^1(\Omega). 
$$
The variational formulation for the four-field linear thermo-poroelasticity problem reads as: for almost every $t\in (0, \tau]$, find $(\bm u,~\xi,~p,~T)\in\bm v\times Q\times W\times W$ such that
\begin{equation}\label{TP_Model_var}
\begin{aligned} 
2\mu(\bm\varepsilon(\bm u), \bm\varepsilon(\bm v)) - (\nabla\cdot\bm v,\xi)&=(\bm f,\bm v)       ,~\forall \bm v\in\bm V,  \\
  -(\nabla\cdot\bm u,\phi)-\tfrac{1}{\lambda}(\xi,\phi)+\tfrac{\alpha}{\lambda}(p,\phi)+\tfrac{\beta}{\lambda}(T,\phi)&=0 ,~\forall\phi\in Q,\\
 -\tfrac\alpha\lambda( \xi_t ,q)+(c_0+\tfrac{\alpha^2}{\lambda})( p_t ,q) + (\tfrac{\alpha\beta}{\lambda}-b_0)( T_t ,q)
  +(\bm K\nabla p,\nabla q)&=(g,q),~\forall q\in W,\\
 -\tfrac\beta\lambda( \xi_t ,S)+(\tfrac{\alpha\beta}{\lambda}-b_0)( p_t ,S)+(a_0+\tfrac{\beta^2}{\lambda})( T_t ,S)   
  +(\bm\Theta\nabla T,\nabla S)&=( H_{s},S), ~\forall S\in W.\\
\end{aligned}
\end{equation}  
For this problem, the initial condition for $\xi$ is
\begin{equation}\label{eq: initial xi0}
 \xi(\cdot,0) = \xi^0 = -\lambda \nabla \cdot \bm{u}^0 + \alpha p^0 + \beta T^0. 
\end{equation}
Moreover, the initial conditions for other variables, \eqref{eq: I_C}, \eqref{eq: initial xi0}, and boundary condition \eqref{eq: B_C} are also applied to problem \eqref{TP_Model_var}.

In the remaining part of this paper, for convenience, we introduce the following parameter transformations:  
\[
c_\alpha = \left(c_0 + \tfrac{\alpha^2}{\lambda}\right), ~ c_{\alpha\beta} = \left(\tfrac{\alpha \beta}{\lambda} - b_0\right), ~ c_\beta = \left(a_0 + \tfrac{\beta^2}{\lambda}\right).
\]  
Through a transformation process, the corresponding equations are converted as follows.
\begin{equation}\label{TP_Model_var after parameters transformation}
\begin{aligned} 
2\mu(\bm\varepsilon(\bm u), \bm\varepsilon(\bm v)) - (\nabla\cdot\bm v,\xi)&=(\bm f,\bm v)       ,~\forall \bm v\in\bm V,  \\
  -(\nabla\cdot\bm u,\phi)-\tfrac{1}{\lambda}(\xi,\phi)+\tfrac{\alpha}{\lambda}(p,\phi)+\tfrac{\beta}{\lambda}(T,\phi)&=0 ,~\forall\phi\in Q,\\
 -\tfrac\alpha\lambda( \xi_t ,q)+c_\alpha( p_t ,q) + c_{\alpha\beta}( T_t ,q)
  +(\bm K\nabla p,\nabla q)&=(g,q),~\forall q\in W,\\
 -\tfrac\beta\lambda( \xi_t ,S)+c_{\alpha\beta}( p_t ,S)+c_\beta( T_t ,S)   
  +(\bm\Theta\nabla T,\nabla S)&=( H_{s},S), ~\forall S\in W.\\
\end{aligned}
\end{equation}  

\section{Numerical algorithms  }\label{sec: algorithms}
Let $h$ denote the maximum diameter of the mesh elements, and for $h > 0$, let $\mathcal{T}_h$ be a family of triangulations of the domain $\Omega$ consisting of triangular elements. The triangulations are assumed to be shape-regular and quasi-uniform. The discrete approximation spaces for $\bm{u}$ and $\xi$ are denoted by $\bm{V}_h \times Q_h$, and they are required to satisfy the following stability (inf-sup) condition:
\begin{equation}\label{infsup for dis}
\begin{aligned}
\inf_{\xi\in Q_h}\sup_{\bm v\in \bm{V}_h}\frac{(\nabla\cdot\bm v,\xi)}{|\bm v|_ {H^1(\Omega)}|\xi|_{L^2(\Omega)}}\ge \gamma>0,
\end{aligned}
\end{equation}
where $\gamma$ is independent of $h$, indicating that $\bm{V}_h \times Q_h$ forms a stable Stokes pair. As a concrete example, we employ Taylor–Hood elements for $(\bm{u}, \xi)$, specifically $(\bm{\mathcal P}_k, \mathcal P_{k-1})$ Lagrange finite elements, along with $\mathcal P_l$ Lagrange finite elements for both the fluid pressure $p$ and temperature $T$. The finite element spaces defined on \(\mathcal{T}_h\) are as follows.
\[
\begin{aligned}
&\bm{V}_h = \left\{ \bm{v}_h \in [H_{0,\Gamma_d}^1(\Omega)]^d \cap \bm{C}^0(\bar{\Omega}) : \bm{v}_h|_K \in \mathcal{P}_k(K), \, \forall K \in \mathcal{T}_h \right\}, \\
&Q_h = \left\{ \phi_h \in L^2(\Omega) \cap C^0(\bar{\Omega}) : \phi_h|_K \in \mathcal{P}_{k-1}(K), \, \forall K \in \mathcal{T}_h \right\}, \\
&W_h = \left\{ q_h \in H_0^1(\Omega) \cap C^0(\bar{\Omega}) : q_h|_K \in \mathcal{P}_l(K), \, \forall K \in \mathcal{T}_h \right\}.
\end{aligned}
\]
 
For time discretization, we consider an equidistant partition consisting of a series of points from $0$ to $\tau$ with a constant step size of $\Delta t$. Specifically, we define $\bm u^n$ as $\bm u(t_n)$, where $t_n$ represents the $n$-th point in the partition. Similarly, we define $\xi^n=\xi(t_n)$ , $p^n=p(t_n)$ and $T^n=T(t_n)$.
The four-field formulation \eqref{TP_Model_var after parameters transformation}, when discretized in time using the backward Euler method, produces the following fully coupled algorithm.

\subsection{The coupled algorithm}
Suppose that an initial value $(\bm u_h^0,~\xi_h^0,~p_h^0,~T_h^0)\in \bm V_h\times Q_h\times W_h\times W_h $ is provided.
 Using the backward Euler method to \eqref{TP_Model_var after parameters transformation}, for the given finite element spaces $\bm V_h \subset \bm V$, $Q_h\subset Q$, $W_h\subset W $, the coupled algorithm reads as: find $(\bm u_h^{n+1},~\xi_h^{n+1},~p_h^{n+1},~T_h^{n+1})\in\bm V_h\times Q_h\times W_h\times W_h$
such that $\forall$ $(\bm v_h ,~\phi_h ,~q_h ,~S_h )\in\bm V_h\times Q_h\times W_h\times W_h$, 
\begin{subequations}
\begin{equation}\label{TP_Model_dis_a}
\begin{aligned} 
2\mu( \bm\varepsilon(\bm u^{n+1}_h),\bm\varepsilon(\bm v_h)) - (\nabla\cdot\bm v_h,\xi^{n+1}_h)&=(\bm f,\bm v_h),  
\end{aligned}
\end{equation}
\begin{equation}\label{TP_Model_dis_b}
\begin{aligned} 
  -(\nabla\cdot\bm u^{n+1}_h,\phi_h)-\tfrac{1}{\lambda}(\xi^{n+1}_h,\phi_h)+\tfrac{\alpha}{\lambda}(p^{n+1}_h,\phi_h)+\tfrac{\beta}{\lambda}(T^{n+1}_h,\phi_h)&=0 ,\\
\end{aligned}
\end{equation}
\begin{equation}\label{TP_Model_dis_c}
\begin{aligned} 
 -  \tfrac\alpha\lambda(\xi_h^{n+1},q_h)+c_\alpha(p_h^{n+1},~q_h)
 + c_{\alpha\beta}(T_h^{n+1},~q_h) 
                +   \Delta t( \bm K\nabla  p^{n+1}_h,~\nabla q_h)&= \\
          c_\alpha(p_h^{n},~q_h)
           + c_{\alpha\beta}(T_h^{n},~q_h)
            -  \tfrac\alpha\lambda(\xi_h^{n},~q_h)
           &+\Delta t(g,~q_h),\\
\end{aligned}
\end{equation}
\begin{equation}\label{TP_Model_dis_d}
\begin{aligned}   
  -\tfrac\beta\lambda(\xi_h^{n+1},S_h)+c_\beta(T_h^{n+1},S_h)
  +c_{\alpha\beta}(p_h^{n+1},~S_h)  
            +\Delta t(\bm\Theta\nabla T^{n+1}_h,~\nabla S_h) &    =         \\
        c_\beta(T_h^{n},~S_h) 
           +c_{\alpha\beta}(p_h^{n},~S_h)    
               -\tfrac\beta\lambda(\xi_h^{n},~S_h)
           +\Delta t& ( H_{s},~S_h).\\
\end{aligned}
\end{equation}
\end{subequations}

The coupled algorithm is unconditionally stable and convergent. Further details on error estimates and stability analysis can be found in \cite{cai2025parameter, cai2025efficient}.

\subsection{Two sequential semi-decoupled algorithms}

Several semi-decoupled algorithms have been developed for Biot’s model \cite{cai2023some, zhao2025optimally}. In these methods, the first time step requires solving all variables simultaneously, while in subsequent steps, the system is partitioned into two subproblems that are solved either sequentially or in parallel, without the need for iteration between them. In this work, we extend these ideas to the thermo-poroelastic model. For the linear thermo-poroelasticity considered here, the initial step is given as follows.

\textbf{Initial step}: set $n=0$ in (\ref{TP_Model_dis_a})-(\ref{TP_Model_dis_d}). That is, 
given \((\bm u_h^0,~\xi_h^0,~p_h^0,~T_h^0)\in\bm V_h\times Q_h\times W_h\times W_h\)
find $(\bm u_h^1,~\xi_h^1,~p_h^1,~T_h^1)\in\bm V_h\times Q_h\times W_h\times W_h$ using (\ref{TP_Model_dis_a})-(\ref{TP_Model_dis_d}).
 
In this work, we explore decoupling strategies to develop semi-decoupled algorithms for the thermo-poroelastic model. The key observation is that for \eqref{TP_Model_dis_a}–\eqref{TP_Model_dis_b}, moving the terms involving $p$ and $T$ to the right-hand side yields a system that only depends on $(\bm u, \xi)$, corresponding to a mixed formulation of the elasticity problem. Once this mixed elasticity subproblem is solved, the updated data can be used to solve a second subproblem involving $p$ and $T$. In this formulation, shifting the terms involving $(\bm u, \xi)$ to the right-hand side transforms \eqref{TP_Model_dis_c} and \eqref{TP_Model_dis_d} into a coupled reaction–diffusion system. Our first semi-decoupled scheme is therefore presented as \textbf{Algorithm 1} below.

\begin{algorithm}[H] 
\caption{ First solve the mixed linear elasticity equations and then solve the coupled reaction–diffusion equations.}
\begin{algorithmic}[1]

\State \textbf{Initial step:} set $n=0$ in (\ref{TP_Model_dis_a})-(\ref{TP_Model_dis_d}).

\For{$n=1, \ldots, N-1$}

    \State \textbf{Step 1. Solving the mixed elasticity equations:} \\
    \hspace{1em} Given \(p^n_h, T^n_h \in W_h\), find 
    \( (\bm u^{n+1}_h, \xi^{n+1}_h) \in \bm V_h \times Q_h \) such that 
    $\forall$ \( (\bm v_h, \phi_h) \in \bm V_h \times Q_h \),
    \begin{subequations}
    \begin{equation}\label{decoupled StR a}
    \begin{aligned} 
    2\mu( \bm\varepsilon(\bm u^{n+1}_h), \bm\varepsilon(\bm v_h))
    - (\nabla \cdot \bm v_h, \xi^{n+1}_h)
    &= (\bm f^{n+1}, \bm v_h),  
    \end{aligned}
    \end{equation}
    \begin{equation}
    \begin{aligned} 
    (\nabla \cdot \bm u^{n+1}_h, \phi_h)
    + \tfrac{1}{\lambda}(\xi^{n+1}_h, \phi_h)
    &= \tfrac{\alpha}{\lambda}(p^n_h, \phi_h)
     + \tfrac{\beta}{\lambda}(T^n_h, \phi_h).
    \end{aligned}
    \end{equation}
    \end{subequations}

    \State \textbf{Step 2. Solving the coupled reaction–diffusion equations:} \\
    \hspace{1em} Find \( (p^{n+1}_h, T^{n+1}_h) \in W_h \times W_h \) such that 
    $\forall$ \( (q_h, S_h) \in W_h \times W_h \),
    \begin{subequations}
    \begin{equation}\label{decoupled StR c}
    \begin{aligned} 
    c_\alpha \Big( \tfrac{p^{n+1}_h - p^n_h}{\Delta t}, q_h \Big)
    + (\bm K \nabla p^{n+1}_h, \nabla q_h)&+ c_{\alpha\beta}\Big( \tfrac{T^{n+1}_h - T^n_h}{\Delta t}, q_h \Big)\\
    &= \tfrac{\alpha}{\lambda}\Big( \tfrac{\xi^{n+1}_h - \xi^n_h}{\Delta t}, q_h \Big)
      + (g^{n+1}, q_h),
    \end{aligned}
    \end{equation}
    \begin{equation}\label{decoupled StR d}
    \begin{aligned}   
    c_\beta\Big( \tfrac{T^{n+1}_h - T^n_h}{\Delta t}, S_h \Big) 
    + (\bm\Theta \nabla T^{n+1}_h, \nabla S_h)  
    &+ c_{\alpha\beta}\Big( \tfrac{p^{n+1}_h - p^n_h}{\Delta t}, S_h \Big)\\
    &= \tfrac{\beta}{\lambda}\Big( \tfrac{\xi^{n+1}_h - \xi^n_h}{\Delta t}, S_h \Big) 
      + (H_{s}^{n+1}, S_h).
    \end{aligned}
    \end{equation}
    \end{subequations}

\EndFor

\end{algorithmic}
\end{algorithm}

Alternatively, after the initial step, one may first solve the coupled reaction--diffusion subproblem 
\eqref{TP_Model_dis_c}--\eqref{TP_Model_dis_d}. 
The newly computed results \((p_h^{n+1}, T_h^{n+1})\) are then used in the subproblem 
\eqref{TP_Model_dis_a}--\eqref{TP_Model_dis_b} to obtain \((\bm u_h^{n+1}, \xi_h^{n+1})\). 
Hence, following a procedure analogous to \textbf{Algorithm~1}, 
we arrive at the next scheme, which we refer to as \textbf{Algorithm~2}.

\begin{algorithm}[H] 
\caption{First solve the coupled reaction--diffusion equations and then solve the mixed elasticity equations.}
\begin{algorithmic}[1]

\State \textbf{Initial step:} set $n=0$ in (\ref{TP_Model_dis_a})-(\ref{TP_Model_dis_d}).

\For{$n=1, \ldots, N-1$}

    \State \textbf{Step 1. Solving the reaction--diffusion equations:} \\
    \hspace{1em} Given \(\xi^{n-1}_h, \xi^n_h \in Q_h\), find 
    \((p^{n+1}_h, T^{n+1}_h) \in W_h \times W_h\) such that 
    for all \((q_h, S_h) \in W_h \times W_h\),
    \begin{subequations}
    \begin{equation}\label{decoupled RtS a}
    \begin{aligned} 
    c_\alpha\Big( \tfrac{p^{n+1}_h-p^n_h}{\Delta t}, q_h \Big)
     + (\bm K \nabla p^{n+1}_h, \nabla q_h) 
    &+ c_{\alpha\beta}\Big( \tfrac{T^{n+1}_h-T^n_h}{\Delta t}, q_h \Big) \\
    & = \tfrac{\alpha}{\lambda}\Big( \tfrac{\xi^n_h-\xi^{n-1}_h}{\Delta t}, q_h \Big) + (g^{n+1}, q_h),
    \end{aligned}
    \end{equation}
    \begin{equation}\label{decoupled RtS b}
    \begin{aligned}   
    c_\beta\Big( \tfrac{T^{n+1}_h-T^n_h}{\Delta t}, S_h \Big)
     + (\bm\Theta \nabla T^{n+1}_h, \nabla S_h)  
    &+ c_{\alpha\beta}\Big( \tfrac{p^{n+1}_h-p^n_h}{\Delta t}, S_h \Big) \\
    & = \tfrac{\beta}{\lambda}\Big( \tfrac{\xi^n_h-\xi^{n-1}_h}{\Delta t}, S_h \Big) + (H_s^{n+1}, S_h).
    \end{aligned}
    \end{equation}
    \end{subequations}

    \State \textbf{Step 2. Solving the mixed elasticity equations:} \\
    \hspace{1em} Find \((\bm u^{n+1}_h, \xi^{n+1}_h) \in \bm V_h \times Q_h\) such that 
    for all \((\bm v_h, \phi_h) \in \bm V_h \times Q_h\),
    \begin{subequations}
    \begin{equation}\label{decoupled RtS c}
    \begin{aligned} 
    2\mu(\bm\varepsilon(\bm u^{n+1}_h), \bm\varepsilon(\bm v_h))
    - (\nabla \cdot \bm v_h, \xi^{n+1}_h)
    &= (\bm f^{n+1}, \bm v_h),
    \end{aligned}
    \end{equation}
    \begin{equation}\label{decoupled RtS d}
    \begin{aligned} 
    (\nabla \cdot \bm u^{n+1}_h, \phi_h)
    + \tfrac{1}{\lambda}(\xi^{n+1}_h, \phi_h)
    &= \tfrac{\alpha}{\lambda}(p^{n+1}_h, \phi_h)
     + \tfrac{\beta}{\lambda}(T^{n+1}_h, \phi_h).
    \end{aligned}
    \end{equation}
    \end{subequations}

\EndFor

\end{algorithmic}
\end{algorithm}

\subsection{A parallel semi-decoupled  algorithm}
Both \textbf{Algorithm~1} and \textbf{Algorithm~2} are inherently sequential, since one subproblem depends on the results of the other. In particular, the second subproblem cannot be solved until the first is completed, which introduces idle time and significantly increases the overall computational cost. This motivates the development of a parallel strategy.

The key observation is as follows. In \textbf{Algorithm~1}, the reaction--diffusion subproblem requires the values 
\((p_h^{n+1}, T_h^{n+1})\) obtained from the mixed elasticity subproblem. However, if the time derivatives 
\(\tfrac{p_h^{n+1}-p_h^n}{\Delta t}\) and \(\tfrac{T_h^{n+1}-T_h^n}{\Delta t}\) are instead approximated using the previous time-step values 
\(\tfrac{p_h^n-p_h^{n-1}}{\Delta t}\) and \(\tfrac{T_h^n-T_h^{n-1}}{\Delta t}\), as in the preceding subsection, then the reaction--diffusion subproblem no longer depends on the mixed elasticity subproblem. Consequently, both subproblems can be solved concurrently at the same time level. A similar modification applies to \textbf{Algorithm~2}. By approximating 
\(\tfrac{\xi_h^{n+1}-\xi_h^n}{\Delta t}\) with 
\(\tfrac{\xi_h^n-\xi_h^{n-1}}{\Delta t}\) in the mixed elasticity subproblem, the dependency is removed and parallel execution becomes possible. Thus, both algorithms lead to the same parallel implementation, i.e., \textbf{Algorithm 3} described below.


\begin{algorithm} [H] 
\caption{  Solve in parallel the mixed elasticity equations and the reaction--diffusion equations.}
\begin{algorithmic}[1]
\State \textbf{Initial step:} set $n=0$ in (\ref{TP_Model_dis_a})-(\ref{TP_Model_dis_d}).
\For{$n=1, \ldots, N-1$} 
\State \textbf{Solving following subproblems in parallel}:
\begin{itemize}
    \item Given \(p^n\), \(T^n\in W_h\), find 
 \( (\bm u^{n+1}_h, ~\xi^{n+1}_h)\in \bm V_h\times Q_h\) such that $\forall$ $(\bm v_h,~\phi_h)\in\bm V_h\times Q_h$
\begin{subequations}
\begin{equation}\label{decoupled SRP a}
\begin{aligned} 
2\mu( \bm\varepsilon(\bm u^{n+1}_h),\bm\varepsilon(\bm v_h)) - (\nabla\cdot\bm v_h,\xi^{n+1}_h)&=(\bm f^{n+1},\bm v_h),  
\end{aligned}
\end{equation}
\begin{equation}\label{decoupled SRP b}
\begin{aligned} 
   (\nabla\cdot\bm u^{n+1}_h,\phi_h)
  +\tfrac{1}{\lambda}(\xi^{n+1}_h,\phi_h)
  &=\tfrac{\alpha}{\lambda}(p^{n}_h,\phi_h)
  +\tfrac{\beta}{\lambda}(T^n_h,\phi_h) .\\
\end{aligned}
\end{equation}
\end{subequations}
    \item Given \(\xi_h^{n-1}\), \(\xi_h^{n}\in Q_h\), 
 find \((p_h^{n+1},~T_h^{n+1})\)\(\in W_h\times W_h\) such that $\forall$ \((q_h,~S_h)\in W_h\times W_h\) 
\begin{subequations}
\begin{equation}\label{decoupled SRP c}
\begin{aligned} 
  c_\alpha( \tfrac{p^{n+1}_h-p^n_h}{\Delta t} ,q_h)
 +\bm (\bm K\nabla p_h^{n+1},\nabla q_h)
 &+ c_{\alpha\beta}( \tfrac{T^{n+1}_h-T^n_h}{\Delta t} ,q_h)\\
              &  =\tfrac\alpha\lambda( \tfrac{\xi^{n}_h-\xi^{n-1}_h}{\Delta t} ,q_h)+(g^{n+1},q_h),
\end{aligned}
\end{equation}
\begin{equation}\label{decoupled SRP d}
\begin{aligned}    
  c_\beta( \tfrac{T^{n+1}_h-T^n_h}{\Delta t} ,S_h) +(\bm\Theta\nabla T^{n+1},\nabla S_h)
 &+c_{\alpha\beta}(  \tfrac{p^{n+1}_h-p^n_h}{\Delta t} ,S_h) \\  
 &=
 \tfrac\beta\lambda( \tfrac{\xi^{n}_h-\xi^{n-1}_h}{\Delta t} ,S_h) +( H_{s}^{n+1},S_h).
\end{aligned}
\end{equation}
\end{subequations}
\end{itemize}
\EndFor

\end{algorithmic}
\end{algorithm}

\begin{remark}\label{re: wellposedness} In all these semi-decoupled algorithms, when solving the coupled reaction--diffusion system in its weak variational form, 
the problem can be equivalently written as the following operator equation:
\[
\begin{bmatrix}
   c_\alpha\bm I - \Delta t \bm K \Delta 
   & c_{\alpha\beta}\bm I \\[6pt]
   c_{\alpha\beta}\bm I 
   & c_\beta\bm I - \Delta t \bm\Theta \Delta
\end{bmatrix}
\begin{bmatrix}
   p_h \\[6pt]
   T_h
\end{bmatrix}
= \bm F,
\]
where $\bm I$ denotes the identity operator and $\bm F$ represents the right-hand side term.  
Under \textbf{Assumptions} (A2)--(A3), the coefficient matrix is symmetric and positive definite.  
Therefore, equations \eqref{decoupled StR c}--\eqref{decoupled StR d} are well-posed.
\end{remark}
\begin{remark}\label{re: innitial step}
We can observe that in \textbf{Step 1} of \textbf{Algorithm 1}, only the data from the 
n-th step are utilized, rendering the initial step unnecessary. However, to maintain consistency with \textbf{Algorithm 2} and \textbf{Algorithm 3}, as well as to facilitate comparative analysis of error data, the initial step is retained.

\end{remark}

\section{Convergence analysis}\label{sec: convergent}

In this section, we present the convergence analysis and derive error estimates for all semi-decoupled algorithms.  
We begin by imposing the following regularity assumptions.  

Assume that 
\[
\bm{u} \in L^{\infty}(J;\bm{H}_{0,\Gamma_d}^{k+1}(\Omega)), \quad 
\partial_t \bm{u} \in L^{2}(J;\bm{H}_{0,\Gamma_d}^{k+1}(\Omega)), \quad 
\partial_{tt} \bm{u} \in L^{2}(J;\bm{H}_{0,\Gamma_d}^{1}(\Omega)),
\]
\[
\xi \in L^{\infty}(J;H^{k}(\Omega)), \quad 
\partial_t \xi \in L^{2}(J;H^{k}(\Omega)), \quad 
\partial_{tt} \xi \in L^{2}(J;L^{2}(\Omega)),
\]
\[
p \in L^{\infty}(J;H_{0}^{l+1}(\Omega)), \quad 
\partial_t p \in L^{2}(J;H_{0}^{l+1}(\Omega)), \quad 
\partial_{tt} p \in L^{2}(J;L^{2}(\Omega)),
\]
\[
T \in L^{\infty}(J;H_{0}^{l+1}(\Omega)), \quad 
\partial_t T \in L^{2}(J;H_{0}^{l+1}(\Omega)), \quad 
\partial_{tt} T \in L^{2}(J;L^{2}(\Omega)).
\]

Next, we introduce suitable projection operators.  
For the discrete finite element spaces \( \bm{V}_h \), \( Q_h \), and \( W_h \) defined in Section~\ref{sec: algorithms}, we define projections
\[
\Pi^{\bm{V}_h}: \bm{V} \to \bm{V}_h, \quad 
\Pi^{Q_h}: Q \to Q_h, \quad 
\Pi^{W_{p,h}}: W \to W_h, \quad 
\Pi^{W_{T,h}}: W \to W_h,
\]
which satisfy the following relations for all 
\((\bm v_h,\phi_h,p_h,T_h)\in \bm V_h\times Q_h\times W_h\times W_h\):
\begin{equation}\label{Pi V}
   2\mu(\bm\varepsilon(\Pi^{\bm V_h}\bm u ),\bm\varepsilon( \bm v_h ))
   - (\nabla\cdot\bm v_h,\Pi^{Q_h} \xi)
   = 2\mu(\bm\varepsilon( \bm u ),\bm\varepsilon( \bm v_h ))
   - (\nabla\cdot\bm v_h, \xi),
\end{equation} 
\begin{equation}\label{Pi Q}
   (\nabla\cdot\Pi^{\bm V_h}\bm u,\phi_h) = (\nabla\cdot \bm u,\phi_h),
\end{equation} 
\begin{equation}\label{Pi Wp}
   \bm K(\nabla\Pi^{W_{p,h}}p,\nabla q_h) = \bm K(\nabla p,\nabla q_h),
\end{equation} 
\begin{equation}\label{Pi WT}
   \bm \Theta(\nabla\Pi^{W_{T,h}}T,\nabla S_h) = \bm\Theta(\nabla T,\nabla S_h).
\end{equation}
Furthermore, for the operators $\Pi^{\bm V_h}$,~$\Pi^{Q_h} $,~$\Pi^{W_{p,h}}$, and $\Pi^{W_{T,h}}$, if $\bm u\in \bm H_{0,\Gamma_d}^{k+1}(\Omega)$, $\xi\in H^k(\Omega)$, $p,~T\in H_0^{l+1}$ we have following
 properties. 
\begin{equation}\label{Pi V Q properties}
   \|\nabla(\Pi^{\bm V_h}\bm u-\bm u)\|_{L^2(\Omega)}+\|\Pi^{Q_h} \xi-\xi\|_{L^2(\Omega)}
   \lesssim h^k\|\bm u\|_{H^{k+1}(\Omega)}+h^k\|\xi\|_{H^{k}(\Omega)}, 
\end{equation} 
\begin{equation}\label{Pi Wp properties}
   \| \Pi^{W_{p,h,}}p-p \|_{L^2(\Omega)}+h\|\nabla(\Pi^{W_{p,h,}}p-p)\|_{L^2(\Omega)}
   \lesssim h^{l+1} \|p\|_{H^{l+1}(\Omega)},
\end{equation} 
and
\begin{equation}\label{Pi WT properties}
     \| \Pi^{W_{T,h}}T-T \|_{L^2(\Omega)}+h\|\nabla(\Pi^{W_{T,h}}T-T)\|_{L^2(\Omega)}
   \lesssim h^{l+1} \|T\|_{H^{l+1}(\Omega)}.
\end{equation}

We also need to cite the following two lemmas, which are frequently used in our subsequent analysis. These two lemmas are easy to obtain, and their proofs can also be found in \cite{gu2023priori}.
\begin{lemma}\label{le: B(u,v)}
    Let \(B\) be a symmetric bilinear form. Then \[     
    2B(u,u-v)=B(u,u)-B(v,v)+B(u-v,u-v),     \] is satisfied, which immediately implies the following inequality \[
    2B(u,u-v)\ge B(u,u)-B(v,v).
    \]
\end{lemma}
\begin{lemma}\label{le: Taylor}
    Let \(f\) be a function that has \(k + 1\) continuous derivatives on an
open interval \( (a,b) \). For any \(t_0 ,t \in (a,b)\), there holds
\[
f(t) = f(t_0) + f'(t_0)(t - t_0) + \cdots + \tfrac{f^{(k)}(t_0)}{k!}(t - t_0)^k + \tfrac{1}{k!} \int_{t_0}^{t} f^{(k+1)}(s)(t - s)^k ds. 
\]
Then, the following estimate for the $L^2$-norm of the last term holds,
    \[
\left\| \tfrac{1}{k!} \int_{t_0}^t f^{(k+1)}(s)(t-s)^k ds \right\|_{L^2(\Omega)}^2 \lesssim (b-a)^{2k+1} \left| \int_{t_0}^t \| f^{(k+1)} \|_{L^2(\Omega)}^2 ds \right|.  
    \]
\end{lemma}

All semi-decoupled algorithm share the same initial step, for which the following error estimates hold for $(\bm{u}_h^1,~\xi_h^1,~p_h^1,~T_h^1)$. These estimates follow directly from Theorem 4.2 for the coupled algorithm presented in \cite{cai2025efficient}.
\begin{lemma}\label{le: couple N=1}
Let \((\bm u,\,\xi,\,p,\,T)\) denote the solution of \eqref{TP_Model_var}, and let \((\bm u_h^1,\,\xi_h^1,\,p_h^1,\,T_h^1)\) be the solution of (\ref{TP_Model_dis_a})–(\ref{TP_Model_dis_d}) at \(n=0\). Then, the following estimate holds:
   \begin{equation}\label{}
    \begin{aligned} 
    & 
       \mu\Vert\bm\varepsilon(e_{\bm u}^{h,1})\Vert_{L^2(\Omega)}^2  
      +\tfrac{c_0-b_0}{2} \Vert  e_p^{h,1}\Vert_{L^2(\Omega)}^2  
       +\tfrac{a_0-b_0}{2} \Vert  e_T^{h,1}\Vert_{L^2(\Omega)}^2 
                      \\%
     & +\Delta t(\bm K\nabla e_p^{h,1},\nabla e_p^{h,1})
        +\Delta t(\bm K\nabla e_T^{h,1},\nabla e_T^{h,1})\\                 
    \le &  \tilde C_1\Delta t^2 + \tilde C_2 h^{2k} + \tilde C_3 h^{2l+2},  
    \end{aligned}
    \end{equation} 
     where
     \begin{equation*}  
    \begin{aligned} 
    &  \tilde C_1   =\tilde C\left(  \int_{0}^{\tau}\| \bm u_{tt}  \|_{H^1(\Omega)}^2 
                  +\int_{0}^{\tau}\| \xi_{tt}  \|_{L^2(\Omega)}^2 
                  +\int_{0}^{\tau}\| p_{tt}  \|_{L^2(\Omega)}^2 
                +\int_{0}^{\tau}\| T_{tt}  \|_{L^2(\Omega)}^2\right),\\ 
    &  \tilde C_2   =\tilde C \left( \int_{0}^{\tau}\| \xi_t  \|_{H^2(\Omega)}^2  
               +\int_{0}^{\tau}\| \bm u_t  \|_{H^2(\Omega)}^2\right),  \\  
    &  \tilde C_3   =\tilde C \left(
               \int_{0}^{\tau}\| p_t  \|_{H^2(\Omega)}^2 
               +  \int_{0}^{\tau}\| T_t  \|_{H^2(\Omega)}^2\right),  
    \end{aligned}
    \end{equation*} 
    where  \(\tilde C=\tilde C(\alpha,\beta,\lambda,a_0,b_0,c_0)\). 
\end{lemma}
\subsection{A priori estimate for the parallel semi-decoupled algorithm}
The error analyses of \textbf{Algorithm~1}, \textbf{Algorithm~2}, and \textbf{Algorithm~3} follow a similar procedure.  
In this subsection, we focus on deriving the \emph{a priori} error estimates for \textbf{Algorithm~3}, the parallel semi-decoupled scheme.

To carry out the analysis, in the remainder of this article, we denote by 
$e_{\bm u}^n,\, e_{\xi}^n,\, e_{p}^n,\, e_{T}^n$ the differences between the exact solutions of \eqref{TP_Model_var} at time step $n$ and the corresponding numerical solutions obtained by \textbf{Algorithm~3}.  
We then decompose each error term into two parts, namely,
\begin{equation}\label{eq: e=eI+eh}
\begin{aligned} 
 &e_{\bm u}^n:=\bm u^n-\bm u_h^n=( \bm u^n-\Pi^{\bm V_h}\bm u^n)+(\Pi^{\bm V_h}\bm u^n -\bm u_h^n):=e_{\bm u}^{I,n}+e_{\bm u}^{h,n},\\
 &e_{\xi}^n:=\xi^n-\xi_h^n=(\xi^n-\Pi^{Q_h}\xi^n)+(\Pi^{Q_h}\xi^n -\xi_h^n):=e_{\xi}^{I,n}+e_{\xi}^{h,n},\\ 
 &e_{p}^n:=p^n-p_h^n=(p^n-\Pi^{W_{p,h}} p^n)+(\Pi^{W_{p,h}} p^n -p_h^n):=e_{p}^{I,n}+e_{p}^{h,n},\\ 
 &e_{T}^n:=T^n-T_h^n=(T^n-\Pi^{W_{p,h}} T^n)+(\Pi^{W_{p,h}} T^n -T_h^n):=e_{T}^{I,n}+e_{T}^{h,n}.\\
\end{aligned}
\end{equation} 
We also denote
\begin{equation}\label{diff_real_num D}
\begin{aligned} 
     (d_{\bm u}^{n+1},~d_{\xi}^{n+1},~d_{p}^{n+1},~d_{T}^{n+1}):=
     (e_{\bm u}^{h,n+1},~e_{\xi}^{h,n+1},~e_{p}^{h,n+1},~e_{T}^{h,n+1})
    -(e_{\bm u}^{h,n},~e_{\xi}^{h,n},~e_{p}^{h,n},~e_{T}^{h,n}).
\end{aligned}
\end{equation}

We first give a lemma that will be used in the following error estimate.
\begin{lemma}\label{le: D estimate} 
       Let \((\bm u,~\xi,~p,~T)\) and \((\bm u_h^{n+1},~\xi_h^{n+1},~p_h^{n+1},~T_h^{n+1})\) for \(n\ge1\) be the solutions of Eqs. \eqref{TP_Model_var} and Eqs. \eqref{decoupled SRP a}-\eqref{decoupled SRP d}, respectively. There holds
    \begin{equation}\label{ineq: estimate sum D}
    \begin{aligned} 
    & \sum_{n=1}^N\Big[
       2\mu\Vert\bm\varepsilon(d_{\bm u}^{n+1})\Vert_{L^2(\Omega)}^2  
      +(c_0-b_0) \Vert d_{p}^{n+1}\Vert_{L^2(\Omega)}^2  
       +(a_0-b_0 )\Vert d_{T}^{n+1}\Vert_{L^2(\Omega)}^2 \\
    &   +b_0\|  d_{p}^{n+1}
           - d_{T}^{n+1} \|_{L^2(\Omega)}^2\big]
     +  \Delta t(\bm K\nabla e_p^{h,N+1},\nabla  e_p^{h,N+1})\\
    & + \Delta t (\bm\Theta\nabla e_T^{h,N+1},\nabla  e_T^{h,N+1}) \\
    \le &  \Delta t^3 \tilde C_1+ h^{2k}\Delta t \tilde C_2 
            +  h^{2l+2}\Delta t \tilde C_3,  
    \end{aligned}
    \end{equation} 
     where
     \begin{equation*}  
    \begin{aligned} 
    &  \tilde C_1   =\tilde C\left(  \int_{0}^{\tau}\| \bm u_{tt}  \|_{H^1(\Omega)}^2 
                  +\int_{0}^{\tau}\| \xi_{tt}  \|_{L^2(\Omega)}^2 
                  +\int_{0}^{\tau}\| p_{tt}  \|_{L^2(\Omega)}^2 
                +\int_{0}^{\tau}\| T_{tt}  \|_{L^2(\Omega)}^2\right),\\ 
    &  \tilde C_2   =\tilde C \left( \int_{0}^{\tau}\| \xi_t  \|_{H^k(\Omega)}^2  
               +\int_{0}^{\tau}\| \bm u_t  \|_{H^k(\Omega)}^2\right),  \\  
    &  \tilde C_3   =\tilde C \left(
               \int_{0}^{\tau}\| p_t  \|_{H^{l+1}(\Omega)}^2 
               +  \int_{0}^{\tau}\| T_t  \|_{H^{l+1}(\Omega)}^2\right),  
    \end{aligned}
    \end{equation*} 
where  \(\tilde C=\tilde C(\alpha,\beta,\lambda,a_0,b_0,c_0)\).    
\end{lemma}
\begin{proof}
By subtracting \eqref{TP_Model_var} from \eqref{decoupled SRP a}–\eqref{decoupled SRP d} at the \((n+1)\)-th time step, we obtain 
\begin{equation}\label{ }
\begin{aligned} 
2\mu(\bm\varepsilon(e_{\bm u}^{n+1}), \bm\varepsilon(\bm v_h)) 
                        - (\nabla\cdot\bm v_h,e_{\xi}^{n+1})&=0,  \\
  -(\nabla\cdot e_{\bm u}^{n+1} ),\phi_h)
  -\tfrac{1}{\lambda}(e_{\xi}^{n+1} ,\phi_h)
  +\tfrac{\alpha}{\lambda}(p^{n+1} -p_h^n ,\phi_h)
  +\tfrac{\beta}{\lambda}(T^{n+1} -T_h^n ,\phi_h)&=0 ,\\
 -\tfrac\alpha\lambda( \xi_t^{n+1} - \tfrac{\xi^{n}_h-\xi^{n-1}_h}{\Delta t} ,q_h)
 +c_\alpha( p_t^{n+1}-\tfrac{p^{n+1}_h-p^{n}_h}{\Delta t} ,q_h)
 +\bm (\bm K\nabla e_p^{n+1},\nabla q_h)&\\
 + c_{\alpha\beta}( T_t^{n+1}-\tfrac{T^{n+1}_h-T^{n}_h}{\Delta t} ,q_h)   &=0,\\
 -\tfrac\beta\lambda( \xi^{n+1}_t- \tfrac{\xi^{n}_h-\xi^{n-1}_h}{\Delta t},S_h)
 +c_\beta( T_t^{n+1}-\tfrac{T^{n+1}_h-T^{n}_h}{\Delta t} ,S_h) 
  +(\bm\Theta\nabla e_T^{n+1},\nabla S_h)&\\
  +c_{\alpha\beta}( p_t^{n+1}-\tfrac{p^{n+1}_h-p^{n}_h}{\Delta t} ,S_h)&=0.\\
\end{aligned}
\end{equation}  
Using the projection operators introduced in \eqref{Pi V}, \eqref{Pi Q}, \eqref{Pi Wp}, and \eqref{Pi WT}, we rearrange the terms and reformulate the above equations as follows: 
\begin{subequations}
\begin{equation}\label{eq: evun+1-evxin+1=0}
\begin{aligned} 
2\mu(\bm\varepsilon(e_{\bm u}^{h,n+1}), \bm\varepsilon(\bm v_h)) 
                        - (\nabla\cdot\bm v_h,e_{\xi}^{h,n+1})&=0,  \\
\end{aligned}
\end{equation}  
\begin{equation}\label{eq: -eun+1-exin+1+pn+1+Tn+1=0}
\begin{aligned} 
    -(\nabla\cdot e_{\bm u}^{h,n+1} ),\phi_h)
  -\tfrac{1}{\lambda}(e_{\xi}^{n+1} ,\phi_h)
  +\tfrac{\alpha}{\lambda}(p^{n+1} -p_h^n ,\phi_h)
  +\tfrac{\beta}{\lambda}(T^{n+1} -T_h^n ,\phi_h)&=0 ,\\
\end{aligned}
\end{equation} 
\begin{equation}\label{eq: reaction dis Dt a}
\begin{aligned} 
 &c_\alpha( d_{p}^{n+1},q_h)
 + \Delta t(\bm K\nabla e_p^{h,n+1},\nabla q_h)
 + c_{\alpha\beta}(d_{T}^{n+1},q_h) \\  
 &=\tfrac\alpha\lambda( \Delta t\xi_t^{n+1} - (\xi^{n}_h-\xi^{n-1}_h) ,q_h) 
  +c_\alpha( \Pi^{W_{p,h}} p^{n+1}-\Pi^{W_{p,h}} p^{n}-\Delta tp_t^{n+1},q_h) \\
 & + c_{\alpha\beta}(\Pi^{W_{T,h}} T^{n+1}-\Pi^{W_{T,h}} T^{n}-\Delta tT_t^{n+1},q_h),\\
 \end{aligned}
\end{equation}  
\begin{equation}\label{eq: reaction dis Dt b}
\begin{aligned} 
& c_\beta(d_{T}^{n+1},S_h)   
  +\Delta t(\bm\Theta\nabla e_T^{h,n+1},S_h)
  +c_{\alpha\beta}( d_{p}^{n+1},S_h)\\
 &= \tfrac\beta\lambda( \Delta t\xi_t^{n+1}- (\xi^{n}_h-\xi^{n-1}_h),S_h) 
   +c_{\alpha\beta}(\Pi^{W_{p,h}} p^{n+1}-\Pi^{W_{p,h}} p^{n}-\Delta tp_t^{n+1} ,S_h) \\
&   +c_\beta(\Pi^{W_{T,h}} T^{n+1}-\Pi^{W_{T,h}} T^{n}-\Delta tT_t^{n+1},S_h) .\\
\end{aligned}
\end{equation}  
\end{subequations}
By subtracting the $n$-th time-step of equation \eqref{eq: evun+1-evxin+1=0} from the $(n+1)$-th time-step of the equation, we derive
\begin{equation}\label{eq: Dvun+1-Dvxin+1=0}
\begin{aligned} 
2\mu(\bm\varepsilon(d_{\bm u}^{n+1}), \bm\varepsilon(\bm v_h)) 
                        - (\nabla\cdot\bm v_h,d_{\xi}^{n+1})&=0. \\
\end{aligned}
\end{equation}  
Performing the same operations on equation \eqref{eq: -eun+1-exin+1+pn+1+Tn+1=0}, we obtain
\begin{equation}\label{eq: -eun+1-exin+1+epn+1+eTn+1}
\begin{aligned} 
  -(\nabla\cdot d_{\bm u}^{n+1},\phi_h)
  -\tfrac{1}{\lambda}(e_{\xi}^{n+1}-e_{\xi}^{n},\phi_h)
  +\tfrac{\alpha}{\lambda}(p^{n+1}-p^n-p_h^n+p_h^{n-1},\phi_h)&\\
  +\tfrac{\beta}{\lambda}(T^{n+1}-T^n-T_h^n+T_h^{n-1},\phi_h)&=0 .\\
\end{aligned}
\end{equation} 
By using the definitions of \(d_{\xi}^{n+1} \), \(d_{p}^{n+1}  \) and \(d_{T}^{n+1}  \), we reformulate \eqref{eq: -eun+1-exin+1+epn+1+eTn+1} as
\begin{equation}\label{eq: Dun+1+Dxin+1-Dpn+1-DTn+1 = etc}
\begin{aligned} 
  &(\nabla\cdot d_{\bm u}^{n+1},\phi_h)
  +\tfrac{1}{\lambda}(d_{\xi}^{n+1},\phi_h)
  -\tfrac{\alpha}{\lambda}(d_{p}^{n},\phi_h)
  -\tfrac{\beta}{\lambda}(d_{T}^{n},\phi_h)\\
   =&-\tfrac{1}{\lambda}(e_{\xi}^{I,n+1}-e_{\xi}^{I,n},\phi_h)
   +\tfrac{\alpha}{\lambda}(p^{n+1}-p^n,\phi_h) 
   -\tfrac{\alpha}{\lambda}(\Pi^{W_{p,h}} p^n-\Pi^{W_{p,h}} p^{n-1},\phi_h)\\
   &+\tfrac{\beta}{\lambda}(T^{n+1}-T^n,\phi_h)
   -\tfrac{\beta}{\lambda}(\Pi^{W_{T,h}} T^n-\Pi^{W_{T,h}} T^{n-1},\phi_h).
\end{aligned}
\end{equation}  
From the second equation of \eqref{TP_Model_var}, we have
\begin{equation}\label{eq: -eun+1-exin+1+epn+1+eTn+1=0}
\begin{aligned}  
  -(\nabla\cdot(\bm u^{n+1}-\bm u^{n})),\phi_h)
  -\tfrac{1}{\lambda}(\xi^{n+1}-\xi^{n},\phi_h)
  +\tfrac{\alpha}{\lambda}(p^{n+1}-p^n,\phi_h)
  +\tfrac{\beta}{\lambda}(T^{n+1}-T^n,\phi_h) =0 ,\\ 
\end{aligned}
\end{equation}  
and
\begin{equation}\label{eq: Dt -eun+1-exin+1+epn+1+eTn+1}
\begin{aligned}  
  -(\nabla\cdot(\Delta t\bm u_t^{n+1})),\phi_h)
  -\tfrac{1}{\lambda}(\Delta t\xi_t^{n+1},\phi_h)
  +\tfrac{\alpha}{\lambda}(\Delta t p_t^{n+1},\phi_h)
  +\tfrac{\beta}{\lambda}(\Delta t T_t^{n+1},\phi_h) =0 .\\ 
\end{aligned}
\end{equation}  
Combining  \eqref{eq: Dun+1+Dxin+1-Dpn+1-DTn+1 = etc}, \eqref{eq: -eun+1-exin+1+epn+1+eTn+1=0}, and \eqref{eq: Dt -eun+1-exin+1+epn+1+eTn+1}, we have
\begin{equation}\label{eq: Dun+1+Dxin+1-Dpn+1-DTn+1 = Pi etc}
\begin{aligned} 
  &(\nabla\cdot d_{\bm u}^{n+1},\phi_h)
  +\tfrac{1}{\lambda}(d_{\xi}^{n+1},\phi_h)
  -\tfrac{\alpha}{\lambda}(d_{p}^{n},\phi_h)
  -\tfrac{\beta}{\lambda}(d_{T}^{n},\phi_h)\\
 =& (\nabla\cdot(\bm u^{n+1}-\bm u^{n}-\Delta t\bm u_t^{n+1})),\phi_h)
   +\tfrac{1}{\lambda}(\xi^{n+1}-\xi^{n}-\Delta t\xi_t^{n+1},\phi_h)
   -\tfrac{1}{\lambda}(e_{\xi}^{I,n+1}-e_{\xi}^{I,n},\phi_h) \\
  &-\tfrac{\alpha}{\lambda}(\Pi^{W_{p,h}} p^n-\Pi^{W_{p,h}} p^{n-1}-\Delta t p_t^{n+1},\phi_h)
   -\tfrac{\beta}{\lambda}(\Pi^{W_{T,h}} T^n-\Pi^{W_{T,h}} T^{n-1}-\Delta t T_t^{n+1},\phi_h).
\end{aligned}
\end{equation}  
By choosing \(\bm v_h=d_{\bm u}^{n+1}\), \(\phi_h=d_{\xi}^{n+1}\), \(q_h=d_{p}^{n+1}\), \(S_h=d_{T}^{n+1}\) in \eqref{eq: Dvun+1-Dvxin+1=0}, \eqref{eq: Dun+1+Dxin+1-Dpn+1-DTn+1 = Pi etc}, \eqref{eq: reaction dis Dt a}, and \eqref{eq: reaction dis Dt b} respectively, we rearrange the terms to obtain the following:
\begin{equation}\label{eq: eDun+1-eDxin+1=0 1}
2\mu(\bm\varepsilon(d_{\bm u}^{n+1}), \bm\varepsilon(d_{\bm u}^{n+1})) 
                        - (\nabla\cdot d_{\bm u}^{n+1},d_{\xi}^{n+1})=0, 
\end{equation}  
\begin{equation}\label{eq: DDun+1+DDxin+1-DDpn+1-DDTn+1 = Pi etc}
\begin{aligned} 
  &(\nabla\cdot d_{\bm u}^{n+1},d_{\xi}^{n+1})
  +\tfrac{1}{\lambda}(d_{\xi}^{n+1},d_{\xi}^{n+1})
  -\tfrac{\alpha}{\lambda}(d_{p}^{n},d_{\xi}^{n+1})
  -\tfrac{\beta}{\lambda}(d_{T}^{n},d_{\xi}^{n+1})\\
  =&  (\nabla\cdot(\bm u^{n+1}-\bm u^{n}-\Delta t\bm u_t^{n+1})),d_{\xi}^{n+1})
   +\tfrac{1}{\lambda}(\xi^{n+1}-\xi^{n}-\Delta t\xi_t^{n+1},d_{\xi}^{n+1})\\
   & -\tfrac{1}{\lambda}(e_{\xi}^{I,n+1}
   -e_{\xi}^{I,n},d_{\xi}^{n+1}) 
    -\tfrac{\alpha}{\lambda}(\Pi^{W_{p,h}} p^n-\Pi^{W_{p,h}} p^{n-1}-\Delta t p_t^{n+1},d_{\xi}^{n+1})\\
   &-\tfrac{\beta}{\lambda}(\Pi^{W_{T,h}} T^n-\Pi^{W_{T,h}} T^{n-1}-\Delta t T_t^{n+1},d_{\xi}^{n+1}) ,
\end{aligned}
\end{equation}  
\begin{equation}\label{eq: reaction dis Dt DD a}
\begin{aligned}  
 &c_\alpha( d_{p}^{n+1},d_{p}^{n+1})
 + \Delta t(\bm K\nabla e_p^{h,n+1},\nabla d_{p}^{n+1})
 + c_{\alpha\beta}(d_{T}^{n+1},d_{p}^{n+1}) \\  
 =&\tfrac\alpha\lambda( \Delta t\xi_t^{n+1} - (\xi^{n}_h-\xi^{n-1}_h) ,d_{p}^{n+1}) 
  +c_\alpha( \Pi^{W_{p,h}} p^{n+1}-\Pi^{W_{p,h}} p^{n}-\Delta tp_t^{n+1},d_{p}^{n+1}) \\
 & + c_{\alpha\beta}(\Pi^{W_{T,h}} T^{n+1}-\Pi^{W_{T,h}} T^{n}-\Delta tT_t^{n+1},d_{p}^{n+1}) ,
 \\ 
 \end{aligned}
\end{equation}
and
\begin{equation}\label{eq: reaction dis Dt DD b}
\begin{aligned}  
 & c_\beta(d_{T}^{n+1},d_{T}^{n+1})   
  +\Delta t(\bm\Theta\nabla e_T^{h,n+1},\nabla d_{T}^{n+1})
  +c_{\alpha\beta}( d_{p}^{n+1},d_{T}^{n+1})\\
 =& \tfrac\beta\lambda( \Delta t\xi_t^{n+1}- (\xi^{n}_h-\xi^{n-1}_h),d_{T}^{n+1}) 
   +c_{\alpha\beta}(\Pi^{W_{p,h}} p^{n+1}-\Pi^{W_{p,h}} p^{n}-\Delta tp_t^{n+1} ,d_{T}^{n+1}) \\
&   +c_\beta(\Pi^{W_{T,h}} T^{n+1}-\Pi^{W_{T,h}} T^{n}-\Delta tT_t^{n+1},d_{T}^{n+1}) .\\
\end{aligned}
\end{equation}  
Next, by summing \eqref{eq: eDun+1-eDxin+1=0 1}, \eqref{eq: DDun+1+DDxin+1-DDpn+1-DDTn+1 = Pi etc}, 
\eqref{eq: reaction dis Dt DD a}, and \eqref{eq: reaction dis Dt DD b} over the time index \(n=1,\dots,N\), 
and then applying the definitions of the projection operators together with an appropriate rearrangement of terms, we obtain
\begin{equation}\label{eq: sum = Eis}  
\begin{aligned}  
&\sum_{n=1}^{N}\Big[ 2\mu\|\bm\varepsilon(d_{\bm u}^{n+1})\|^2_{L^2(\Omega)}
+\tfrac{1}{\lambda}\|d_{\xi}^{n+1} \|^2_{L^2(\Omega)} \\ 
&+ \tfrac{1}{\lambda}\|\alpha d_{p}^{n+1}+\beta  d_{T}^{n+1}\|^2_{L^2(\Omega)} 
+ (c_0-b_0) \| d_{p}^{n+1}\|^2_{L^2(\Omega)} 
 + (a_0-b_0) \| d_{T}^{n+1}\|^2_{L^2(\Omega)} \\
&  +b_0  \|d_{p}^{n+1}-d_{T}^{n+1}\|^2_{L^2(\Omega)} 
  +\Delta t(\bm K\nabla e_p^{h,n+1},\nabla d_{p}^{n+1})
  +\Delta t(\bm\Theta\nabla e_T^{h,n+1},\nabla d_{T}^{n+1})
   \Big] \\
= 
 &\sum_{n=1}^{N}\Big[ 2\mu(\bm\varepsilon(d_{\bm u}^{n+1}), \bm\varepsilon(d_{\bm u}^{n+1}))
 +\tfrac{1}{\lambda}(d_{\xi}^{n+1},d_{\xi}^{n+1}) \\
&  +c_\alpha( d_{p}^{n+1} ,d_{p}^{n+1})
  +(a_0+\tfrac{ \beta^2}{\lambda})( d_{T}^{n+1} ,d_{T}^{n+1})   
  + \Delta t(\bm K\nabla e_p^{h,n+1},\nabla d_{p}^{n+1})\\
& +c_{\alpha\beta}( d_{p}^{n+1},d_{T}^{n+1})
  + c_{\alpha\beta}(d_{T}^{n+1},d_{p}^{n+1}) 
   +\Delta t(\bm\Theta\nabla e_T^{h,n+1},\nabla d_{T}^{n+1})\Big] \\
   \\
= & \sum_{i=1}^{11} E_i  , 
\end{aligned}
\end{equation}  
where 
\begin{equation*} 
\begin{aligned}  
&E_1=\sum_{n=1}^{N}(\nabla\cdot(\bm u^{n+1}-\bm u^{n}-\Delta t\bm 
                                        u_t^{n+1})),d_{\xi}^{n+1}),\\
\end{aligned}
\end{equation*}  
\begin{equation*} 
\begin{aligned}                                          
&E_2=\sum_{n=1}^{N} \tfrac{1}{\lambda}(\Pi^{Q_h}\xi^{n+1}-\Pi^{Q_h}\xi^{n}-\Delta t\xi_t^{n+1},d_{\xi}^{n+1}),\\
\end{aligned}
\end{equation*} 
\begin{equation*} 
\begin{aligned}  
&E_3=\sum_{n=1}^{N}-\tfrac{\alpha}{\lambda}(\Pi^{W_{p,h}} p^n-\Pi^{W_{p,h}} p^{n-1}-\Delta t p_t^{n+1},d_{\xi}^{n+1}),\\
\end{aligned}
\end{equation*} 
\begin{equation*} 
\begin{aligned}  
&E_4=\sum_{n=1}^{N}-\tfrac{\beta}{\lambda}(\Pi^{W_{T,h}} T^n-\Pi^{W_{T,h}} T^{n-1}-\Delta t  T_t^{n+1},d_{\xi}^{n+1}), \\
\end{aligned}
\end{equation*} 
\begin{equation*} 
\begin{aligned}  
&E_5=\sum_{n=1}^{N}\tfrac\alpha\lambda
( \Delta t\xi_t^{n+1}-\Pi^{Q_h} \xi^{n}+\Pi^{Q_h}\xi^{n-1},d_{p}^{n+1}),
\\
&E_6=\sum_{n=1}^{N}c_\alpha( \Pi^{W_{p,h}} p^{n+1}-\Pi^{W_{p,h}} p^{n}-\Delta tp_t^{n+1},d_{p}^{n+1}),\\
&E_7=\sum_{n=1}^{N}c_{\alpha\beta}(\Pi^{W_{T,h}} T^{n+1}-\Pi^{W_{T,h}} T^{n} -\Delta tT_t^{n+1},d_{p}^{n+1}) ,\\
&E_8=\sum_{n=1}^{N}\tfrac\beta\lambda
( \Delta t\xi_t^{n+1}-\Pi^{Q_h}\xi^{n}+\Pi^{Q_h}\xi^{n-1},d_{T}^{n+1}),\\
&E_9=\sum_{n=1}^{N}c_{\alpha\beta}(\Pi^{W_{p,h}} p^{n+1}-\Pi^{W_{p,h}} 
                   p^{n}-\Delta tp_t^{n+1} ,d_{T}^{n+1}),\\
&E_{10}=\sum_{n=1}^{N}c_\beta(\Pi^{W_{T,h}} T^{n+1}-\Pi^{W_{T,h}} T^{n}- 
                   \Delta tT_t^{n+1},d_{T}^{n+1}), \\
\end{aligned}
\end{equation*} 
and
\begin{equation*} 
\begin{aligned}  
&E_{11}=\sum_{n=1}^{N} 
    \tfrac{1}{\lambda}(\alpha d_{p}^{n}+\beta d_{T}^{n},d_{\xi}^{n+1})
    +\sum_{n=1}^{N} 
    \tfrac{1}{\lambda}(\alpha d_{p}^{n+1}+\beta d_{T}^{n+1}
    ,d_{\xi}^{n}). 
\end{aligned}
\end{equation*} 

Next, by applying Lemma \ref{le: Taylor},  the Cauchy-Schwarz inequality,
 and Young’s inequality, we can bound $E_1$ as:
\begin{equation*} 
\begin{aligned}  
E_1&=\sum_{n=1}^{N}(\nabla\cdot(\bm u^{n+1}-\bm u^{n}-\Delta t\bm 
               u_t^{n+1})),d_{\xi}^{n+1})\\ 
&\le 
\sum_{n=1}^{N} \tfrac{C}{\eta_1} \| \bm u^{n+1} -\bm u^{n} - \Delta t 
     \bm u_t^{n+1} \|_{H^1(\Omega)}^2 
+ \ \sum_{n=1}^{N}\eta_1 \| d_{\xi}^{n+1} \|_{L^2(\Omega)}^2 \\
&=  \sum_{n=1}^{N}\tfrac{C}{\eta_1} \| \int_{t_{n}}^{t_{n+1}}\bm u_{tt}(s) (s-t_{n}) ds \|_{H^1(\Omega)}^2 
+ \sum_{n=1}^{N}\eta_1\| d_{\xi}^{n+1} \|_{L^2(\Omega)}^2 \\
&\le  \tfrac{C}{\eta_1}\Delta t^3 \int_{0}^{\tau} \left\| \bm u_{tt} \right\|_{H^1(\Omega)}^2 
+  \sum_{n=1}^{N} \eta_1\| d_{\xi}^{n+1} \|_{L^2(\Omega)}^2,
\end{aligned}
\end{equation*} 
where $\eta_1$ is a constant which will be determined later.
By applying Lemma \ref{le: Taylor}, the Cauchy-Schwarz inequality, Young’s inequality, and the properties of the operators in \eqref{Pi V Q properties}-\eqref{Pi WT properties}, we can bound $E_2$ as:
 \begin{equation*} 
\begin{aligned}  
E_2&=\sum_{n=1}^{N} \tfrac{1}{\lambda}(\Pi^{Q_h}\xi^{n+1}-\Pi^{Q_h}\xi^{n}-\Delta t\xi_t^{n+1},d_{\xi}^{n+1})\\ 
 =&\sum_{n=1}^N\tfrac1\lambda ( \Pi^{Q_h} \xi^{n+1}-\Pi^{Q_h} \xi^{n}-(\xi^{n+1}-\xi^{n})
         +(\xi^{n+1}-\xi^{n})-\Delta t \xi_t^{n+1} , d_{\xi}^{n+1})\\
  \le&
   \tfrac1\lambda \Big(
   h^{2k}
     \sum_{n=1}^N \tfrac{C}{\eta_1}\|\xi^{n+1}-\xi^{n}  \|_{H^{k+1}(\Omega)}^2
   +
 \sum_{n=1}^N \tfrac{C}{\eta_1} \| \xi^{n+1}-\xi^{n}-\Delta t\xi_t^{n+1} \|_{L^2(\Omega)}^2
  \\
    &  +\sum_{n=1}^N \eta_1 \| d_{\xi}^{n+1} \|_{L^2(\Omega)}^2
   \Big)    \\    
  \le & \tfrac{1}{\lambda}\Big(
  h^{2k}\Delta t \tfrac{C}{\eta_1} 
    \int_{0}^{\tau}\| \xi_t  \|_{H^k(\Omega)}^2
   + \Delta t^3\tfrac{C}{\eta_1}  \int_{0}^{\tau}\| \xi_{tt}  \|_{L^2(\Omega)}^2 + \sum_{n=1}^N \eta_1 \| d_{\xi}^{n+1} \|_{L^2(\Omega)}^2
   \Big),                      
\end{aligned}
\end{equation*} 
Similarly, \(E_3\), \(E_4\), \(E_5\), \(E_6\), \(E_7\), \(E_8\), \(E_9\) and \(E_{10}\) can be bounded as follows:
 \begin{equation*} 
\begin{aligned}  
E_3
=&\sum_{n=1}^{N}\Big[-\tfrac{\alpha}{\lambda}(\Pi^{W_{p,h}} p^n-\Pi^{W_{p,h}} p^{n-1}-\Delta t p_t^{n}
  +(\Delta t p_t^{n}-\Delta t p_t^{n+1}),d_{\xi}^{n+1})
  \Big]\\
\le &
\tfrac\alpha\lambda\Big(h^{2l+2}\Delta t \tfrac{C}{\eta_1} 
    \int_{0}^{\tau}\| p_t  \|_{H^{l+1}(\Omega)}^2
   + \Delta t^3 \tfrac{C}{\eta_1}  \int_{0}^{\tau}\| p_{tt}  \|_{L^2(\Omega)}^2
    +  \sum_{n=1}^N  \eta_1 \| d_{\xi}^{n+1} \|_{L^2(\Omega)}^2 \Big), 
\end{aligned}
\end{equation*} 
 \begin{equation*} 
\begin{aligned}  
E_4\le&
\tfrac\beta\lambda\Big(h^{2l+2}\Delta t\tfrac{C}{\eta_1}
    \int_{0}^{\tau}\| T_t  \|_{H^{l+1}(\Omega)}^2
   + \Delta t^3 \tfrac{C}{\eta_1} \int_{0}^{\tau}\| T_{tt}  \|_{L^2(\Omega)}^2
    + \sum_{n=1}^N  \eta_1 \| d_{\xi}^{n+1} \|_{L^2(\Omega)}^2 \Big), \\
\end{aligned}
\end{equation*} 
\begin{equation*} 
\begin{aligned}   
 E_5\le&
\tfrac\alpha\lambda\Big(
    h^{2k}\Delta t \tfrac{C}{\eta_1}\
    \int_{0}^{\tau} \| \xi_t  \|_{H^{k+1}(\Omega)}^2
    +\Delta t^3 \tfrac{C}{\eta_1}
    \int_{0}^{\tau}\| \xi_t  \|_{H^{k}(\Omega)}^2
 + \sum_{n=1}^N  \eta_1 \| d_{p}^{n+1} \|_{L^2(\Omega)}^2\\
&+\Delta t^3 \tfrac{C}{\eta_1}\int_{0}^\tau \|\xi_{tt}\|_{L^2(\Omega)}^2
\Big),\\
\end{aligned}
\end{equation*} 
 \begin{equation*} 
\begin{aligned}  
E_6
 \le&
  c_\alpha\left(h^{2l+2}\Delta t
    \tfrac{C}{\eta_1}\int_{0}^{\tau}\| p_t  \|_{H^{l+1}(\Omega)}^2
   + \Delta t^3  \tfrac{C}{\eta_1}\int_{0}^{\tau}\| p_{tt}  \|_{L^2(\Omega)}^2
    +  \sum_{n=1}^N  \eta_1\| d_{p}^{n+1} \|_{L^2(\Omega)}^2
    \right),
\end{aligned}
\end{equation*} 
 \begin{equation*} 
\begin{aligned}   
 E_7
  \le&
  \left|c_{\alpha\beta}|\right|\left(h^{2l+2}\Delta t \tfrac{C}{\eta_1}
    \int_{0}^{\tau}\| T_t  \|_{H^{l+1}(\Omega)}^2
   + \Delta t^3 \tfrac{C}{\eta_1} \int_{0}^{\tau}\| T_{tt}  \|_{L^2(\Omega)}^2
    + \sum_{n=1}^N  \eta_1\| d_{p}^{n+1} \|_{L^2(\Omega)}^2
    \right),\\
\end{aligned}
\end{equation*} 
\begin{equation*} 
\begin{aligned}   
E_8\le&
\tfrac\beta\lambda\Big( 
    h^{2k}\Delta t \tfrac{C}{\eta_1} 
    \int_{0}^{\tau} \| \xi_t  \|_{H^{k+1}(\Omega)}^2
    +\Delta t^3 \tfrac{C}{\eta_1}
    \int_{0}^{\tau}\| \xi_t  \|_{H^{k}(\Omega)}^2
 + \sum_{n=1}^N  \eta_1 \| d_{T}^{n+1} \|_{L^2(\Omega)}^2\\
&+\Delta t^3 \tfrac{C}{\eta_1}\int_{0}^\tau \|\xi_{tt}\|_{L^2(\Omega)}^2
\Big),\\
\end{aligned}
\end{equation*} 
 \begin{equation*} 
\begin{aligned}   
E_9  \le&\left|c_{\alpha\beta}|\right|\left(h^{2l+2}\Delta t \tfrac{C}{\eta_1} 
    \int_{0}^{\tau}\| p_t  \|_{H^{l+1}(\Omega)}^2
   + \Delta t^3  \tfrac{C}{\eta_1} \int_{0}^{\tau}\| p_{tt}  \|_{L^2(\Omega)}^2
    +  \sum_{n=1}^N \eta_1 \| d_{T}^{n+1} \|_{L^2(\Omega)}^2
    \right),\\
\end{aligned}
\end{equation*} 
and
 \begin{equation*} 
\begin{aligned}   
E_{10} \le &c_\beta\left(h^{2l+2}\Delta t \tfrac{C}{\eta_1} 
    \int_{0}^{\tau}\| T_t  \|_{H^{l+1}(\Omega)}^2
   + \Delta t^3 \tfrac{C}{\eta_1}   \int_{0}^{\tau}\| T_{tt}  \|_{L^2(\Omega)}^2
    + \sum_{n=1}^N \eta_1 \| d_{T}^{n+1} \|_{L^2(\Omega)}^2
    \right).\\
\end{aligned}
\end{equation*}
By Young’s inequality and Cauchy-Schwarz inequality, we obtain the bound of \(E_{11}\) as
\begin{equation*} 
\begin{aligned}  
 E_{11} 
 \le &  \sum_{n=1}^N\tfrac{1}{2\lambda}\| \alpha d_{p}^{n}
           +\beta d_{T}^{n} \|_{L^2(\Omega)}^2
       + \sum_{n=1}^N\tfrac{1}{2\lambda}\| d_{\xi}^{n+1} \|_{L^2(\Omega)}^2    \\
     &+ \sum_{n=1}^N\tfrac{1}{2\lambda}\| \alpha d_{p}^{n+1}
           +\beta d_{T}^{n+1} \|_{L^2(\Omega)}^2
       + \sum_{n=1}^N\tfrac{1}{2\lambda}\| d_{\xi}^{n} \|_{L^2(\Omega)}^2    \\
\le &
   \sum_{n=1}^N\tfrac{1}{\lambda}\| d_{\xi}^{n+1} \|_{L^2(\Omega)}^2 
  +\sum_{n=1}^N\tfrac{1}{\lambda}\| \alpha d_{p}^{n+1}
           +\beta d_{T}^{n+1} \|_{L^2(\Omega)}^2 .\\
\end{aligned}
\end{equation*} 
By the Lemma \ref{le: B(u,v)}, we have
\begin{equation} \label{eq: epDp get epn+1-epn}
\begin{aligned} 
 \Delta t(\bm K\nabla e_p^{h,n+1},\nabla d_{p}^{n+1})
  \ge  \tfrac{\Delta t}{2}(\bm K\nabla e_p^{h,n+1},\nabla e_p^{h,n+1})
    -  \tfrac{\Delta t}{2}(\bm K\nabla e_p^{h,n},\nabla e_p^{h,n})
\end{aligned}
\end{equation} 
and
\begin{equation}\label{eq: eTDT get eTn+1-eTn} 
\begin{aligned} 
 \Delta t(\bm\Theta\nabla e_T^{h,n+1},\nabla d_{T}^{n+1})
 \ge \tfrac{\Delta t}{2}(\bm\Theta\nabla e_T^{h,n+1},\nabla e_T^{h,n+1})
    -\tfrac{\Delta t}{2}(\bm\Theta\nabla e_T^{h,n},\nabla e_T^{h,n}) . 
\end{aligned}
\end{equation} 
By combining \eqref{eq: sum = Eis}, \eqref{eq: epDp get epn+1-epn}, and \eqref{eq: eTDT get eTn+1-eTn}, together with the established bounds for \(E_1\) through \(E_{11}\), we obtain the following.
\begin{equation}
\begin{aligned}  
&\sum_{n=1}^{N}\Big[ 2\mu\|\bm\varepsilon(d_{\bm u}^{n+1})\|^2_{L^2(\Omega)}
+\tfrac{1}{\lambda}\|d_{\xi}^{n+1} \|^2_{L^2(\Omega)} \\ 
&+ \tfrac{1}{\lambda}\|\alpha d_{p}^{n+1}+\beta  d_{T}^{n+1}\|^2_{L^2(\Omega)} 
+ (c_0-b_0) \| d_{p}^{n+1}\|^2_{L^2(\Omega)} 
 + (a_0-b_0) \| d_{T}^{n+1}\|^2_{L^2(\Omega)} \\
&  +b_0  \|d_{p}^{n+1}-d_{T}^{n+1}\|^2_{L^2(\Omega)}  \Big] 
  +\tfrac{\Delta t}{2}(\bm K\nabla e_p^{h,N+1},\nabla e_p^{h,N+1})
  +\tfrac{\Delta t}{2}(\bm\Theta\nabla e_T^{h,N+1},\nabla e_T^{h,N+1})
  \\
\le \mathbb C\Bigg(&
        \sum_{n=1}^N\eta_1 \| d_{\xi}^{n+1} \|_{L^2(\Omega)}^2
      + \sum_{n=1}^N \eta_1\| d_{p}^{n+1} \|_{L^2(\Omega)}^2
      + \sum_{n=1}^N\eta_1\| d_{T}^{n+1} \|_{L^2(\Omega)}^2\\
   +&\Delta t^3 \tfrac{C}{\eta_1} \left(  \int_{0}^{\tau}\| \bm u_{tt}  \|_{H^1(\Omega)}^2 
                  +\int_{0}^{\tau}\| \xi_{tt}  \|_{L^2(\Omega)}^2 
                  +\int_{0}^{\tau}\| p_{tt}  \|_{L^2(\Omega)}^2 
                +\int_{0}^{\tau}\| T_{tt}  \|_{L^2(\Omega)}^2\right)\\ 
   +&h^{2k} \Delta t \tfrac{C}{\eta_1}  \left( \int_{0}^{\tau}\| \xi_t  \|_{H^k(\Omega)}^2  
               +\int_{0}^{\tau}\| \bm u_t  \|_{H^{k+1}(\Omega)}^2\right)  \\  
   + &h^{2l+2} \Delta t \tfrac{C}{\eta_1}  \left(
               \int_{0}^{\tau}\| p_t  \|_{H^{l+1}(\Omega)}^2 
               +  \int_{0}^{\tau}\| T_t  \|_{H^{l+1}(\Omega)}^2\right)
    \Bigg)           \\
    +&  \sum_{n=1}^N\tfrac{1}{\lambda}\| d_{\xi}^{n+1} \|_{L^2(\Omega)}^2 
 +\sum_{n=1}^N\tfrac{1}{\lambda}\| \alpha d_{p}^{n+1}
           +\beta d_{T}^{n+1} \|_{L^2(\Omega)}^2\\
   +&\tfrac{\Delta t}{2}(\bm K\nabla e_p^{h,1},\nabla e_p^{h,1})
    +\tfrac{\Delta t}{2}(\bm\Theta\nabla e_T^{h,1},\nabla e_T^{h,1})            ,
\end{aligned}
\end{equation}  
where \(\mathbb{C}\) denotes a constant depending on \(a_0\), \(b_0\), \(c_0\), \(\alpha\), \(\beta\), and \(\lambda\).  
Furthermore, since the finite element spaces \(\bm V_h\) and \(Q_h\) satisfy the inf–sup condition \eqref{infsup for dis}, we can combine this property with the Cauchy–Schwarz inequality and the definition of the strain operator \(\bm\varepsilon\)  to deduce from \eqref{eq: evun+1-evxin+1=0} the following relations:
\begin{equation} \label{ineq: applying infsup d}
\begin{aligned} 
 \|d_{\xi}^{n+1}\|_{L^2(\Omega)}^2\le 
 \sup_{v_h\in \bm V_h} \tfrac{(\nabla\cdot \bm v_h,d_{\xi}^{n+1})}{\|\bm v_h\|_{H^1(\Omega)}}
 =\sup_{v_h\in \bm V_h}
 \tfrac{2\mu(\bm\varepsilon(d_{\bm u}^{n+1}), \bm\varepsilon(\bm v_h))}{\|\bm v_h\|_{H^1(\Omega)}}
 \le \tfrac{2\mu}{\gamma}  \|\bm\varepsilon(d_{\bm u}^{n+1})\|_{L^2(\Omega)}^2.
\end{aligned}
\end{equation}  
From Lemma \ref{le: couple N=1}, together with equation \eqref{ineq: applying infsup d} and by choosing a sufficiently small parameter $\eta_1$, we obtain \eqref{ineq: estimate sum D}.
\end{proof}

We next present the following theorem on error estimates, which establishes the optimal convergence of \textbf{Algorithm 3}.

\begin{theorem}[estimate for \textbf{Algorithm 3}]\label{th: error estimate}
    Let \((\bm u,~\xi,~p,~T)\) and \((\bm u_h^{n+1},~\xi_h^{n+1},~p_h^{n+1},~T_h^{n+1})\) for \(n\ge1\) be the solutions of Eqs. \eqref{TP_Model_var} and Eqs. \eqref{decoupled SRP a}-\eqref{decoupled SRP d}, respectively. There holds
    \begin{equation}\label{}
    \begin{aligned} 
    & 
       \mu\Vert\bm\varepsilon(e_{\bm u}^{h,N+1})\Vert_{L^2(\Omega)}^2  
       +\Vert  e_\xi^{h,N+1}\Vert_{L^2(\Omega)}^2 
      +\tfrac{c_0-b_0}{2} \Vert  e_p^{h,N+1}\Vert_{L^2(\Omega)}^2  
       +\tfrac{a_0-b_0}{2} \Vert  e_T^{h,N+1}\Vert_{L^2(\Omega)}^2 
                      \\
                       &+  \Delta t \sum_{n=1}^N \big(  
            (\bm K\nabla e_p^{h,n+1},\nabla e_p^{h,n+1})  
        + (\bm\Theta\nabla e_T^{h,n+1},\nabla e_T^{h,n+1}) \big)\\
    \le &  \tilde C_1\Delta t^2 + \tilde C_2 h^{2k} + \tilde C_3 h^{2l+2},  
    \end{aligned}
    \end{equation} 
     where
     \begin{equation*}  
    \begin{aligned} 
    &  \tilde C_1   =\tilde C\left(  \int_{0}^{\tau}\| \bm u_{tt}  \|_{H^1(\Omega)}^2 
                  +\int_{0}^{\tau}\| \xi_{tt}  \|_{L^2(\Omega)}^2 
                  +\int_{0}^{\tau}\| p_{tt}  \|_{L^2(\Omega)}^2 
                +\int_{0}^{\tau}\| T_{tt}  \|_{L^2(\Omega)}^2\right),\\ 
    &  \tilde C_2   =\tilde C \left( \int_{0}^{\tau}\| \xi_t  \|_{H^k(\Omega)}^2  
               +\int_{0}^{\tau}\| \bm u_t  \|_{H^{k+1}(\Omega)}^2\right), \\  
    &  \tilde C_3   =\tilde C \left(
               \int_{0}^{\tau}\| p_t  \|_{H^{l+1}(\Omega)}^2 
               +  \int_{0}^{\tau}\| T_t  \|_{H^{l+1}(\Omega)}^2\right),  
    \end{aligned}
    \end{equation*} 
where  \(\tilde C=\tilde C(\alpha,\beta,\lambda,a_0,b_0,c_0)\).    
\end{theorem}
\begin{proof}

By choosing \(\bm v_h=d_{\bm u}^{n+1}\), \(\phi_h=e_{\xi}^{h,n+1}\), \(q_h=e_p^{h,n+1}\), \(S_h=e_T^{h,n+1}\) in \eqref{eq: evun+1-evxin+1=0}, \eqref{eq: Dun+1+Dxin+1-Dpn+1-DTn+1 = Pi etc}, \eqref{eq: reaction dis Dt a}, and \eqref{eq: reaction dis Dt b} respectively, we rearrange the terms to obtain the following:
\begin{equation}\label{eq: eDun+1-eDxin+1=0}
2\mu(\bm\varepsilon(e_{\bm u}^{h,n+1}), \bm\varepsilon(d_{\bm u}^{n+1})) 
                        - (\nabla\cdot d_{\bm u}^{n+1},e_{\xi}^{h,n+1})=0, 
\end{equation}  
\begin{equation}\label{eq: eDun+1+eDxin+1-eDpn+1-eDTn+1 = Pi etc}
\begin{aligned} 
  &(\nabla\cdot d_{\bm u}^{n+1},e_{\xi}^{h,n+1})
  +\tfrac{1}{\lambda}(d_{\xi}^{n+1},e_{\xi}^{h,n+1})
  -\tfrac{\alpha}{\lambda}(d_{p}^{n},e_{\xi}^{h,n+1})
  -\tfrac{\beta}{\lambda}(d_{T}^{n},e_{\xi}^{h,n+1})\\
 =& (\nabla\cdot(\bm u^{n+1}-\bm u^{n}-\Delta t\bm u_t^{n+1})),e_{\xi}^{h,n+1})
   +\tfrac{1}{\lambda}(\xi^{n+1}-\xi^{n}-\Delta t\xi_t^{n+1},e_{\xi}^{h,n+1})\\
 &  -\tfrac{1}{\lambda}(e_{\xi}^{I,n+1}-e_{\xi}^{I,n},e_{\xi}^{h,n+1}) 
  -\tfrac{\alpha}{\lambda}(\Pi^{W_{p,h}} p^n-\Pi^{W_{p,h}} p^{n-1}-\Delta t p_t^{n+1},e_{\xi}^{h,n+1})\\
 &  -\tfrac{\beta}{\lambda}(\Pi^{W_{T,h}} T^n-\Pi^{W_{T,h}} T^{n-1}-\Delta t T_t^{n+1},e_{\xi}^{h,n+1}) ,
\end{aligned}
\end{equation}  
\begin{equation}\label{eq: reaction dis Dt De a}
\begin{aligned}  
 &c_\alpha( d_{p}^{n+1},e_p^{h,n+1})
 + \Delta t(\bm K\nabla e_p^{h,n+1},\nabla e_p^{h,n+1})
 + c_{\alpha\beta}(d_{T}^{n+1},e_p^{h,n+1}) \\  
 =&\tfrac\alpha\lambda( \Delta t\xi_t^{n+1} - (\xi^{n}_h-\xi^{n-1}_h) ,e_p^{h,n+1}) 
  +c_\alpha( \Pi^{W_{p,h}} p^{n+1}-\Pi^{W_{p,h}} p^{n}-\Delta tp_t^{n+1},e_p^{h,n+1}) \\
 & + c_{\alpha\beta}(\Pi^{W_{T,h}} T^{n+1}-\Pi^{W_{T,h}} T^{n})-\Delta tT_t^{n+1},e_p^{h,n+1}) ,
 \\ 
 \end{aligned}
\end{equation}
and
\begin{equation}\label{eq: reaction dis Dt De b}
\begin{aligned}  
 & c_\beta(d_{T}^{n+1},e_T^{h,n+1})   
  +\Delta t(\bm\Theta\nabla e_T^{h,n+1},\nabla e_T^{h,n+1})
  +c_{\alpha\beta}( d_{p}^{n+1},e_T^{h,n+1})\\
 =& \tfrac\beta\lambda( \Delta t\xi_t^{n+1}- (\xi^{n}_h-\xi^{n-1}_h),e_T^{h,n+1}) 
   +c_{\alpha\beta}(\Pi^{W_{p,h}} p^{n+1}-\Pi^{W_{p,h}} p^{n}-\Delta tp_t^{n+1} ,e_T^{h,n+1}) \\
&   +c_\beta(\Pi^{W_{T,h}} T^{n+1}-\Pi^{W_{T,h}} T^{n}-\Delta tT_t^{n+1},e_T^{h,n+1}) 
 .\\
\end{aligned}
\end{equation}  
Then, we take the summation of \eqref{eq: eDun+1-eDxin+1=0}, \eqref{eq: eDun+1+eDxin+1-eDpn+1-eDTn+1 = Pi etc}, \eqref{eq: reaction dis Dt De a}, and \eqref{eq: reaction dis Dt De b} over the index \(n\) from \(1\) to \(N\) and rearrange the terms to obtain
\begin{equation}\label{ }
\begin{aligned}  
&LHS\\
:= &\sum_{n=1}^{N}\Big[ 2\mu(\bm\varepsilon(e_{\bm u}^{h,n+1}), \bm\varepsilon(d_{\bm u}^{n+1}))
 +\tfrac{1}{\lambda}(d_{\xi}^{n+1},e_{\xi}^{h,n+1})
  -\tfrac{\alpha}{\lambda}(d_{p}^{n+1},e_{\xi}^{h,n+1})
  -\tfrac{\beta}{\lambda}(d_{T}^{n+1},e_{\xi}^{h,n+1}) \\
&  +c_\alpha( d_{p}^{n+1} ,e_p^{h,n+1})
  +(a_0+\tfrac{ \beta^2}{\lambda})( d_{T}^{n+1} ,e_T^{h,n+1})   
  + \Delta t(\bm K\nabla e_p^{h,n+1},\nabla e_p^{h,n+1})\\
& +c_{\alpha\beta}( d_{p}^{n+1},e_T^{h,n+1})
  + c_{\alpha\beta}(d_{T}^{n+1},e_p^{h,n+1}) 
   +\Delta t(\bm\Theta\nabla e_T^{h,n+1},\nabla e_T^{h,n+1})\\
&- \tfrac{\alpha}{\lambda}(d_{\xi}^{n+1},e_{p}^{h,n+1})
   -\tfrac{\beta}{\lambda}(d_{\xi}^{n+1},e_{T}^{h,n+1})   
   \Big] \\
= & \sum_{i=1}^{11} D_i  , 
\end{aligned}
\end{equation}  
where 
\begin{equation*} 
\begin{aligned}  
&D_1=\sum_{n=1}^{N}(\nabla\cdot(\bm u^{n+1}-\bm u^{n}-\Delta t\bm 
                                        u_t^{n+1})),e_{\xi}^{h,n+1}), \\
\end{aligned}
\end{equation*} 
\begin{equation*} 
\begin{aligned}  
&D_2=\sum_{n=1}^{N} \tfrac{1}{\lambda}(\Pi^{Q_h}\xi^{n+1}-\Pi^{Q_h}\xi^{n}-\Delta t\xi_t^{n+1},e_{\xi}^{h,n+1}),\\
&D_3=\sum_{n=1}^{N}-\tfrac{\alpha}{\lambda}(\Pi^{W_{p,h}} p^n-\Pi^{W_{p,h}} p^{n-1}-\Delta t p_t^{n+1},e_{\xi}^{h,n+1}),\\
\end{aligned}
\end{equation*} 
\begin{equation*} 
\begin{aligned}  
&D_4=\sum_{n=1}^{N}-\tfrac{\beta}{\lambda}(\Pi^{W_{T,h}} T^n-\Pi^{W_{T,h}} T^{n-1}-\Delta t  T_t^{n+1},e_{\xi}^{h,n+1}), \\
&D_5=\sum_{n=1}^{N}\tfrac\alpha\lambda
( \Delta t\xi_t^{n+1}-\Pi^{Q_h} \xi^{n}+\Pi^{Q_h}\xi^{n-1},e_p^{h,n+1})
,
\\
&D_6=\sum_{n=1}^{N}c_\alpha( \Pi^{W_{p,h}} p^{n+1}-\Pi^{W_{p,h}} p^{n}-\Delta tp_t^{n+1},e_p^{h,n+1}),\\
\end{aligned}
\end{equation*} 
\begin{equation*} 
\begin{aligned}  
&D_7=\sum_{n=1}^{N}c_{\alpha\beta}(\Pi^{W_{T,h}} T^{n+1}-\Pi^{W_{T,h}} T^{n} -\Delta tT_t^{n+1},e_p^{h,n+1}) ,\\
&D_8=\sum_{n=1}^{N}\tfrac\beta\lambda
( \Delta t\xi_t^{n+1}-\Pi^{Q_h}\xi^{n}+\Pi^{Q_h}\xi^{n-1},e_T^{h,n+1})
,\\
\end{aligned}
\end{equation*} 
\begin{equation*} 
\begin{aligned}  
&D_9=\sum_{n=1}^{N}c_{\alpha\beta}(\Pi^{W_{p,h}} p^{n+1}-\Pi^{W_{p,h}} 
                   p^{n}-\Delta tp_t^{n+1} ,e_T^{h,n+1}),\\
&D_{10}=\sum_{n=1}^{N}c_\beta(\Pi^{W_{T,h}} T^{n+1}-\Pi^{W_{T,h}} T^{n}- 
                   \Delta tT_t^{n+1},e_T^{h,n+1}), \\
\end{aligned}
\end{equation*} 
and
\begin{equation*} 
\begin{aligned}  
D_{11}=&\sum_{n=1}^{N} 
    \tfrac{1}{\lambda}(\alpha d_{p}^{n}+\beta d_{T}^{n},e_{\xi}^{h,n+1})
   +\sum_{n=1}^{N} 
\tfrac{1}{\lambda}(\alpha d_{p}^{n+1}+\beta d_{T}^{n+1},e_{\xi}^{h,n+1})\\
   &+\sum_{n=1}^{N}
   \tfrac\alpha\lambda(d_{\xi}^{n}+d_{\xi}^{n+1},e_p^{h,n+1}) 
   +\sum_{n=1}^{N}
   \tfrac\beta\lambda(d_{\xi}^{n}+d_{\xi}^{n+1},e_T^{h,n+1}). 
\end{aligned}
\end{equation*} 
Rearranging \(LHS\) and applying Lemma \ref{le: B(u,v)}, we can estimate that

\begin{equation}\label{ }
\begin{aligned}  
LHS
= &\sum_{n=1}^{N}\Big[ 2\mu(\bm\varepsilon(e_{\bm u}^{h,n+1}), 
                       \bm\varepsilon(d_{\bm u}^{n+1}))\\
  &+\tfrac1\lambda(d_{\xi}^{n+1}-\alpha d_{p}^{n+1}-\beta d_{T}^{n+1}
                ,\ e_{\xi}^{h,n+1}-\alpha e_p^{h,n+1}-\beta e_T^{h,n+1})
                \\ 
  & + (c_0-b_0) ( d_{p}^{n+1} ,e_p^{h,n+1})
    + (a_0-b_0) ( d_{T}^{n+1} ,e_T^{h,n+1})   
    \\
  &  +  b_0( d_{p}^{n+1}-d_{T}^{n+1} ,e_p^{h,n+1}-e_T^{h,n+1})\\
  &  + \Delta t(\bm K\nabla e_p^{h,n+1},\nabla e_p^{h,n+1})  
    +\Delta t(\bm\Theta\nabla e_T^{h,n+1},\nabla e_T^{h,n+1}) 
       \Big] \\ 
\ge&\sum_{n=1}^{N}\Big[ 
      \mu \|\bm\varepsilon(e_{\bm u}^{h,n+1} )\|^2_{L^2(\Omega)}
              -\mu \|\bm\varepsilon(e_{\bm u}^{h,n} )\|^2_{L^2(\Omega)} \\
  &+\tfrac{1}{2\lambda} \|\ e_{\xi}^{h,n+1}-\alpha e_p^{h,n+1}-\beta 
                      e_T^{h,n+1}\|^2_{L^2(\Omega)}    
   -\tfrac{1}{2\lambda} \|\ e_{\xi}^{h,n}-\alpha e_p^{h,n}-\beta 
                      e_T^{h,n}\|^2_{L^2(\Omega)}  \\
  & +\tfrac{c_0-b_0}{2} \| e_p^{h,n+1} \|^2_{L^2(\Omega)}
                -\tfrac{c_0-b_0}{2}  \| e_p^{h,n} \|^2_{L^2(\Omega)}\\
  & +\tfrac{a_0-b_0}{2} \| e_T^{h,n+1} \|^2_{L^2(\Omega)} 
            -\tfrac{a_0-b_0}{2}  \| e_T^{h,n} \|^2_{L^2(\Omega)} 
    \\
  & + \tfrac{b_0}{2} \| e_p^{h,n+1}-e_T^{h,n+1} \|^2_{L^2(\Omega)} 
           - \tfrac{b_0}{2} \| e_p^{h,n}-e_T^{h,n} \|^2_{L^2(\Omega)} 
           \Big]  \\ 
  &  +\Delta t\sum_{n=1}^{N}\Big[ 
          (\bm K\nabla e_p^{h,n+1},\nabla e_p^{h,n+1})  
        + (\bm\Theta\nabla e_T^{h,n+1},\nabla e_T^{h,n+1})
        \Big],\\ 
\end{aligned}
\end{equation}  
which directly implies that
\begin{equation}\label{ineq: LHS ge etc.}
\begin{aligned}  
LHS 
\ge& \tfrac12\Big[ 
      2\mu \|\bm \varepsilon(e_{\bm u}^{h,N+1}) \|^2_{L^2(\Omega)}
                    -\mu \|\bm\varepsilon(e_{\bm u}^{h,1}  )\|^2_{L^2(\Omega)} \\
  &+\tfrac{1}{\lambda} \|\ e_{\xi}^{h,N+1}-\alpha e_p^{h,N+1}-\beta 
                      e_T^{h,N+1}\|^2_{L^2(\Omega)}    
   -\tfrac{1}{\lambda} \|\ e_{\xi}^{h,1}-\alpha e_p^{h,1}-\beta 
                      e_T^{h,1}\|^2_{L^2(\Omega)}  \\
  & +(c_0-b_0) \| e_p^{h,N+1} \|^2_{L^2(\Omega)}
                -(c_0-b_0)   \| e_p^{h,1} \|^2_{L^2(\Omega)}\\
  & +(a_0-b_0) \| e_T^{h,N+1} \|^2_{L^2(\Omega)} 
            -(a_0-b_0)  \| e_T^{h,1} \|^2_{L^2(\Omega)} 
    \\
  & +  b_0  \| e_p^{h,N+1}-e_T^{h,N+1} \|^2_{L^2(\Omega)} 
           -  b_0  \| e_p^{h,1}-e_T^{h,1} \|^2_{L^2(\Omega)} 
           \Big]  \\ 
  &  +\Delta t\sum_{n=1}^{N}\Big[ 
          (\bm K\nabla e_p^{h,n+1},\nabla e_p^{h,n+1})  
        + (\bm\Theta\nabla e_T^{h,n+1},\nabla e_T^{h,n+1})
        \Big].\\ 
\end{aligned}
\end{equation}  

Next, by applying Lemma \ref{le: Taylor},  the Cauchy-Schwarz inequality,
 and Young’s inequality, we can bound $D_1$ as:
\begin{equation*} 
\begin{aligned}  
D_1&=\sum_{n=1}^{N}(\nabla\cdot(\bm u^{n+1}-\bm u^{n}-\Delta t\bm 
               u_t^{n+1})),e_{\xi}^{h,n+1})\\ 
&\le \tfrac{1}{\Delta t}  \tfrac{C}{\eta_1}
\sum_{n=1}^{N} \| \bm u^{n+1} -\bm u^{n} - \Delta t 
      \bm u_t^{n+1} \|_{H^1(\Omega)}^2 
+ \eta_1\Delta t \sum_{n=1}^{N} \| e_{\xi}^{h,n+1} \|_{L^2(\Omega)}^2 \\
&= \tfrac{1}{\Delta t}  \tfrac{C}{\eta_1}\sum_{n=1}^{N} \| \int_{t_{n}}^{t_{n+1}}\bm u_{tt}(s) (s-t_{n}) ds \|_{H^1(\Omega)}^2 
+ \eta_1\Delta t  \sum_{n=1}^{N} \| e_{\xi}^{h,n+1} \|_{L^2(\Omega)}^2 \\
&\le \Delta t^2  \tfrac{C}{\eta_1}\int_{0}^{\tau} \left\| \bm u_{tt} \right\|_{H^1(\Omega)}^2 ds 
+ \eta_1\Delta t \sum_{n=1}^{N} \| e_{\xi}^{h,n+1} \|_{L^2(\Omega)}^2.
\end{aligned}
\end{equation*} 
By applying Lemma \ref{le: Taylor}, the Cauchy-Schwarz inequality, Young’s inequality, and the properties of the operators in \eqref{Pi V Q properties}-\eqref{Pi WT properties}, we can bound $D_2$ as:
 \begin{equation*} 
\begin{aligned}  
D_2=&\sum_{n=1}^{N} \tfrac{1}{\lambda}(\Pi^{Q_h}\xi^{n+1}-\Pi^{Q_h}\xi^{n}-\Delta t\xi_t^{n+1},e_{\xi}^{h,n+1})\\ 
 =&\sum_{n=1}^N\tfrac1\lambda ( \Pi^{Q_h} \xi^{n+1}-\Pi^{Q_h} \xi^{n}-(\xi^{n+1}-\xi^{n})
         +(\xi^{n+1}-\xi^{n})-\Delta t \xi_t^{n+1} , e_\xi^{h,n+1})\\
  \le&
   \tfrac1\lambda \left(h^{2k}
     \tfrac{1}{\Delta t} \tfrac{C}{\eta_1} \sum_{n=1}^N     \| \xi^{n+1}-\xi^{n}  \|_{H^k(\Omega)}^2
   + \Delta t^2  \tfrac{C}{\eta_1} \int_{t_1}^{\tau}\| \xi_{tt}  \|_{L^2(\Omega)}^2
    + \eta_1\Delta t \sum_{n=1}^N \| e_{\xi}^{h,n+1} \|_{L^2(\Omega)}^2
   \right)    \\    
  \le&
   \tfrac1\lambda \left(h^{2k} \tfrac{C}{\eta_1}
    \int_{0}^{\tau}\| \xi_t  \|_{H^k(\Omega)}^2
   + \Delta t^2  \tfrac{C}{\eta_1} \int_{0}^{\tau}\| \xi_{tt}  \|_{L^2(\Omega)}^2
    + \eta_1\Delta t  \sum_{n=1}^N        \| e_{\xi}^{h,n+1} \|_{L^2(\Omega)}^2
   \right),                      
\end{aligned}
\end{equation*} 
Similarly, \(D_3-D_{10}\) can be bounded as follows.
 \begin{equation*} 
\begin{aligned}  
D_3
=&\sum_{n=1}^{N}\Big[-\tfrac{\alpha}{\lambda}(\Pi^{W_{p,h}} p^n-\Pi^{W_{p,h}} p^{n-1}-\Delta t p_t^{n},e_{\xi}^{h,n+1})
  -(\Delta t p_t^{n}-\Delta t p_t^{n+1},e_{\xi}^{h,n+1})
  \Big]\\
\le&
\tfrac\alpha\lambda\Big(h^{2l+2}\tfrac{C}{\eta_1} 
    \int_{0}^{\tau}\| p_t  \|_{H^{l+1}(\Omega)}^2
   + \Delta t^2 \tfrac{C}{\eta_1}  \int_{0}^{\tau}\| p_{tt}  \|_{L^2(\Omega)}^2
    + \eta_1\Delta t \sum_{n=1}^N        \| e_{\xi}^{h,n+1} \|_{L^2(\Omega)}^2
   \Big), 
\end{aligned}
\end{equation*} 
\begin{equation*} 
\begin{aligned}  
D_4\le&
\tfrac\beta\lambda\Big(h^{2l+2}\tfrac{C}{\eta_1} 
    \int_{0}^{\tau}\| T_t  \|_{H^{l+1}(\Omega)}^2
   + \Delta t^2 \tfrac{C}{\eta_1}  \int_{0}^{\tau}\| T_{tt}  \|_{L^2(\Omega)}^2
    + \eta_1\Delta t  \sum_{n=1}^N        \| e_{\xi}^{h,n+1} \|_{L^2(\Omega)}^2 
   \Big), \\
\end{aligned}
\end{equation*} 
\begin{equation*} 
\begin{aligned}   
 D_5\le&
\tfrac\alpha\lambda\Big(h^{2k}\tfrac{C}{\eta_1} 
    \int_{0}^{\tau}\| \xi_t  \|_{H^{k}(\Omega)}^2
   + \Delta t^2 \tfrac{C}{\eta_1}  \int_{0}^{\tau}\| \xi_{tt}  \|_{L^2(\Omega)}^2
    + \eta_1\Delta t \sum_{n=1}^N \| e_{p}^{h,n+1} \|_{L^2(\Omega)}^2
\Big),\\
\end{aligned}
\end{equation*} 
 \begin{equation*} 
\begin{aligned}  
D_6
 \le&
  c_\alpha\left(h^{2l+2}\tfrac{C}{\eta_1} 
    \int_{0}^{\tau}\| p_t  \|_{H^{l+1}(\Omega)}^2
   + \Delta t^2 \tfrac{C}{\eta_1}  \int_{0}^{\tau}\| p_{tt}  \|_{L^2(\Omega)}^2
    + \eta_1\Delta t \sum_{n=1}^N        \| e_{p}^{h,n+1} \|_{L^2(\Omega)}^2
    \right),
\end{aligned}
\end{equation*} 
 \begin{equation*} 
\begin{aligned}   
 D_7
  \le&
  \left|c_{\alpha\beta}|\right|\left(h^{2l+2}\tfrac{C}{\eta_1} 
    \int_{0}^{\tau}\| T_t  \|_{H^{l+1}(\Omega)}^2
   + \Delta t^2 \tfrac{C}{\eta_1}  \int_{0}^{\tau}\| T_{tt}  \|_{L^2(\Omega)}^2
    + \eta_1\Delta t \sum_{n=1}^N        \| e_{p}^{h,n+1} \|_{L^2(\Omega)}^2
    \right),\\
\end{aligned}
\end{equation*} 
\begin{equation*} 
\begin{aligned}   
D_8\le&
\tfrac\beta\lambda\Big(h^{2k}\tfrac{C}{\eta_1} 
    \int_{0}^{\tau}\| \xi_t  \|_{H^{k}(\Omega)}^2
   + \Delta t^2 \tfrac{C}{\eta_1}  \int_{0}^{\tau}\| \xi_{tt}  \|_{L^2(\Omega)}^2
    + \eta_1\Delta t \sum_{n=1}^N \| e_{T}^{h,n+1} \|_{L^2(\Omega)}^2
\Big),\\
\end{aligned}
\end{equation*} 
 \begin{equation*} 
\begin{aligned}   
D_9  \le&\left|c_{\alpha\beta}|\right|\left(h^{2l+2}\tfrac{C}{\eta_1} 
    \int_{0}^{\tau}\| p_t  \|_{H^{l+1}(\Omega)}^2
   + \Delta t^2 \tfrac{C}{\eta_1}  \int_{0}^{\tau}\| p_{tt}  \|_{L^2(\Omega)}^2
    + \eta_1\Delta t  \sum_{n=1}^N   \| e_{T}^{h,n+1} \|_{L^2(\Omega)}^2
    \right),\\
\end{aligned}
\end{equation*} 
and
 \begin{equation*} 
\begin{aligned}   
D_{10} \le&c_\beta\left(h^{2l+2} \tfrac{C}{\eta_1} 
    \int_{0}^{\tau}\| T_t  \|_{H^{l+1}(\Omega)}^2
   + \Delta t^2 \tfrac{C}{\eta_1}   \int_{0}^{\tau}\| T_{tt}  \|_{L^2(\Omega)}^2
    + \eta_1\Delta t \sum_{n=1}^N        \| e_{T}^{h,n+1} \|_{L^2(\Omega)}^2
    \right).\\
\end{aligned}
\end{equation*} 
For \(D_{11}\), by Young's inequality and Cauchy-Schwarz inequality, we give the following estimate:
\begin{equation*} 
\begin{aligned}  
 D_{11}\le& \sum_{n=1}^{N} \tfrac{\alpha^2}{\lambda^2}\tfrac{1}{\Delta t}\tfrac{C}{\eta_1}\| d_{p}^{n+1}\|_{L^2(\Omega)}^2
  + \sum_{n=1}^{N}\tfrac{\beta^2}{\lambda^2}\tfrac{1}{\Delta t}\tfrac{C}{\eta_1}\| d_{T}^{n+1}\|_{L^2(\Omega)}^2\\
 & +\sum_{n=1}^{N}\eta_1\Delta t
   \| e_{\xi}^{h,n+1} \|_{L^2(\Omega)}^2
   +\sum_{n=1}^{N}\eta_1\Delta t
   \| e_{\xi}^{h,n+1} \|_{L^2(\Omega)}^2\\
 &+\sum_{n=1}^{N} \tfrac{\alpha^2}{\lambda^2}\tfrac{1}{\Delta t}\tfrac{C}{\eta_1}
   \| d_{\xi}^{n+1}\|_{L^2(\Omega)}^2
   +\sum_{n=1}^{N}  \eta_1\Delta t 
   \| e_{p}^{h,n+1} \|_{L^2(\Omega)}^2 \\
  &+\sum_{n=1}^{N} \tfrac{\beta^2}{\lambda^2}\tfrac{1}{\Delta t}\tfrac{C}{\eta_1}
   \| d_{\xi}^{n+1}\|_{L^2(\Omega)}^2
   +\sum_{n=1}^{N}\eta_1\Delta t 
   \| e_{T}^{h,n+1} \|_{L^2(\Omega)}^2.
\end{aligned}
\end{equation*} 

Then, according to inequality \eqref{ineq: LHS ge etc.}, the estimates of $D_1$-$D_{11}$, and the equation $LHS=\sum_{i=1}^{11}D_i$, we have
\begin{equation}\label{eq: left le LHS le right+Ds}
\begin{aligned}  
 & 
      2\mu \|\bm\varepsilon(e_{\bm u}^{h,N+1} )\|^2_{L^2(\Omega)}
                       +\tfrac{1}{\lambda} \|\ e_{\xi}^{h,N+1}-\alpha e_p^{h,N+1}-\beta 
                      e_T^{h,N+1}\|^2_{L^2(\Omega)}    
  +(c_0-b_0) \| e_p^{h,N+1} \|^2_{L^2(\Omega)} \\
  & +(a_0-b_0) \| e_T^{h,N+1} \|^2_{L^2(\Omega)}  
   +b_0  \| e_p^{h,N+1}-e_T^{h,N+1} \|^2_{L^2(\Omega)}  
             \\ 
  &  +\Delta t\sum_{n=1}^{N}\Big[ 
          (\bm K\nabla e_p^{h,n+1},\nabla e_p^{h,n+1})  
        + (\bm\Theta\nabla e_T^{h,n+1},\nabla e_T^{h,n+1})
       \Big] \\
\le&\mathbb C
\Bigg(
 h^{2k}\tfrac{C}{ \eta_1} 
     \int_{0}^{\tau}\| \xi_t  \|_{H^{k}(\Omega)}^2   
    +h^{2l+2}\tfrac{C}{ \eta_1} (
    \int_{0}^{\tau}\| p_t  \|_{H^{l+1}(\Omega)}^2
   +  \int_{0}^{\tau}\| T_t  \|_{H^{l+1}(\Omega)}^2 
    )\\
&+ \Delta t^2\tfrac{C}{ \eta_1} (\int_{0}^{\tau} \left\| \bm u_{tt} \right\|_{H^1(\Omega)}^2
             +\int_{0}^{\tau}\| \xi_{tt}  \|_{L^2(\Omega)}^2
             + \int_{0}^{\tau}\| p_{tt}  \|_{L^2(\Omega)}^2
             +\int_{0}^\tau \| T_{tt}  \|_{L^2(\Omega)}^2)
\\
&+ \eta_1\Delta t \sum_{n=1}^{N}\| e_{\xi}^{h,n+1} \|_{L^2(\Omega)}^2
 +\eta_1\Delta t \sum_{n=1}^N \| e_{p}^{h,n+1} \|_{L^2(\Omega)}^2
 +\eta_1\Delta t \sum_{n=1}^N \| e_{T}^{h,n+1} \|_{L^2(\Omega)}^2\\
&+\sum_{n=1}^{N}\tfrac{1}{\Delta t}\tfrac{C}{ \eta_1} 
        (\| d_{\xi}^{n+1} \|_{L^2(\Omega)}^2 
               +\| d_{p}^{n+1}\|_{L^2(\Omega)}^2
       +\| d_{T}^{n+1} \|_{L^2(\Omega)}^2)          
       \\
&+\|e_{\bm u}^{h,1} \|^2_{L^2(\Omega)}
 +\| e_{\xi}^{h,1} \|_{L^2(\Omega)}^2 
  +\| e_{p}^{h,1} \|_{L^2(\Omega)}^2 
  +\| e_{T}^{h,1} \|_{L^2(\Omega)}^2
  \Bigg ),\\
\end{aligned}
\end{equation}  
where $\mathbb{C}$ denotes a constant depending on $a_0$, $b_0$, $c_0$, $\alpha$, $\beta$, and $\lambda$.  
Proceeding analogously to the derivation in \eqref{ineq: applying infsup d}, we obtain the following relations from equation \eqref{eq: evun+1-evxin+1=0}:
\begin{equation}\label{ineq: applying infsup e}
\begin{aligned} 
 \|e_\xi^{h,n+1}\|_{L^2(\Omega)}^2\le 
 \sup_{v_h\in \bm V_h} \tfrac{(\nabla\cdot \bm v_h,e_\xi^{h,n+1})}{\|\bm v_h\|_{H^1(\Omega)}}
 =\sup_{v_h\in \bm V_h}
 \tfrac{2\mu(\bm\varepsilon(e_{\bm u}^{h,n+1}), \bm\varepsilon(\bm v_h))}{\|\bm v_h\|_{H^1(\Omega)}}
 \le \tfrac{2\mu}{\gamma}  \|\bm\varepsilon(e_{\bm u}^{h,n+1})\|_{L^2(\Omega)}^2.
\end{aligned}
\end{equation}  
Bounding the term 
\[ \sum_{n=1}^{N}\tfrac{1}{\Delta t}(\| d_{\xi}^{n+1} \|_{L^2(\Omega)}^2 
               +\| d_{p}^{n+1}\|_{L^2(\Omega)}^2
       +\| d_{T}^{n+1} \|_{L^2(\Omega)}^2)   
\] 
by Lemma \ref{le: D estimate} and \eqref{ineq: applying infsup d}, bounding the term 
\[
\|e_{\bm u}^{h,1} \|^2_{L^2(\Omega)}
 +\| e_{\xi}^{h,1} \|_{L^2(\Omega)}^2 
  +\| e_{p}^{h,1} \|_{L^2(\Omega)}^2 
  +\| e_{T}^{h,1} \|_{L^2(\Omega)}^2
\]
by Lemma \ref{le: couple N=1} and \eqref{ineq: applying infsup e}, we can reformulate \eqref{eq: left le LHS le right+Ds} as:
\begin{equation*}
\begin{aligned}  
 & 
      2\mu \|\bm\varepsilon(e_{\bm u}^{h,N+1} )\|^2_{L^2(\Omega)}
                       +\tfrac{1}{\lambda} \|\ e_{\xi}^{h,N+1}-\alpha e_p^{h,N+1}-\beta 
                      e_T^{h,N+1}\|^2_{L^2(\Omega)}    
  +(c_0-b_0) \| e_p^{h,N+1} \|^2_{L^2(\Omega)} \\
  & +(a_0-b_0) \| e_T^{h,N+1} \|^2_{L^2(\Omega)}  
   +b_0  \| e_p^{h,N+1}-e_T^{h,N+1} \|^2_{L^2(\Omega)}  
             \\ 
  &  +\Delta t\sum_{n=1}^{N}\Big[ 
          (\bm K\nabla e_p^{h,n+1},\nabla e_p^{h,n+1})  
        + (\bm\Theta\nabla e_T^{h,n+1},\nabla e_T^{h,n+1})
       \Big] \\
\le&\mathbb C
\Bigg(
 h^{2k}\tfrac{C}{\eta_1} 
    \int_{0}^{\tau}\| \xi_t  \|_{H^k(\Omega)}^2   
    +h^{2l+2}\tfrac{C}{\eta_1}(
    \int_{0}^{\tau}\| p_t  \|_{H^{l+1}(\Omega)}^2
   +  \int_{0}^{\tau}\| T_t  \|_{H^{l+1}(\Omega)}^2 
    )\\
&+ \Delta t^2\tfrac{C}{\eta_1}(\int_{0}^{\tau} \left\| \bm u_{tt} \right\|_{H^1(\Omega)}^2
             +\int_{0}^{\tau}\| \xi_{tt}  \|_{L^2(\Omega)}^2
             + \int_{0}^{\tau}\| p_{tt}  \|_{L^2(\Omega)}^2
             +\int_{0}^\tau \| T_{tt}  \|_{L^2(\Omega)}^2)
\\
&+ \eta_1\Delta t \sum_{n=1}^{N}\Big(
 2\mu \|\bm\varepsilon(e_{\bm u}^{h,n+1})\|_{L^2(\Omega)}^2
 +\tfrac{1}{\lambda} \|\ e_{\xi}^{h,n+1}-\alpha e_p^{h,n+1}-\beta 
                      e_T^{h,n+1}\|^2_{L^2(\Omega)}  \\
     &   +  (c_0-b_0)\| e_{p}^{h,n+1} \|_{L^2(\Omega)}^2
         +  (a_0-b_0)\| e_{T}^{h,n+1} \|_{L^2(\Omega)}^2
         +b_0  \| e_p^{h,n+1}-e_T^{h,n+1} \|^2_{L^2(\Omega)}    \Big )
         \Bigg ),                  
\end{aligned}
\end{equation*}  
which, by applying the discrete Grönwall inequality and choosing $\eta_1$ sufficiently small, leads to the following result:
\begin{equation*}
\begin{aligned}  
 & 
      2\mu \|\bm\varepsilon(e_{\bm u}^{h,N+1} )\|^2_{L^2(\Omega)}
                       +\tfrac{1}{\lambda} \|\ e_{\xi}^{h,N+1}-\alpha e_p^{h,N+1}-\beta 
                      e_T^{h,N+1}\|^2_{L^2(\Omega)}    
  +(c_0-b_0) \| e_p^{h,N+1} \|^2_{L^2(\Omega)} \\
  & +(a_0-b_0) \| e_T^{h,N+1} \|^2_{L^2(\Omega)}  
   +b_0  \| e_p^{h,N+1}-e_T^{h,N+1} \|^2_{L^2(\Omega)}  
             \\ 
  &  +\Delta t\sum_{n=1}^{N}\Big[ 
          (\bm K\nabla e_p^{h,n+1},\nabla e_p^{h,n+1})  
        + (\bm\Theta\nabla e_T^{h,n+1},\nabla e_T^{h,n+1})
       \Big] \\
\le&\tilde C
\Bigg(
 h^{2k}(
    \int_{0}^{\tau}\| \xi_t  \|_{H^k(\Omega)}^2 
    +\int_{0}^{\tau}\| \xi_t  \|_{H^{k}(\Omega)}^2  
    ) 
    +h^{2l+2}(
    \int_{0}^{\tau}\| p_t  \|_{H^{l+1}(\Omega)}^2
   +  \int_{0}^{\tau}\| T_t  \|_{H^{l+1}(\Omega)}^2 
    )\\
&+ \Delta t^2(\int_{0}^{\tau} \left\| \bm u_{tt} \right\|_{H^1(\Omega)}^2
             +\int_{0}^{\tau}\| \xi_{tt}  \|_{L^2(\Omega)}^2
             + \int_{0}^{\tau}\| p_{tt}  \|_{L^2(\Omega)}^2
             +\int_{0}^\tau \| T_{tt}  \|_{L^2(\Omega)}^2)\Bigg ).           
\end{aligned}
\end{equation*}  
 This yields the desired result. 
\end{proof}

\begin{remark}\label{re: robust}
It is important to note that Theorem \ref{th: error estimate} yields the following inequality: 
$$
\begin{aligned} 
    & 
       \mu\Vert\bm\varepsilon(e_{\bm u}^{h,N+1})\Vert_{L^2(\Omega)}^2  
       +\Vert  e_\xi^{h,N+1}\Vert_{L^2(\Omega)}^2 
                      +  \Delta t \sum_{n=1}^N \big(  
            (\bm K\nabla e_p^{h,n+1},\nabla e_p^{h,n+1})  
        + (\bm\Theta\nabla e_T^{h,n+1},\nabla e_T^{h,n+1}) \big)\\
    \le &  \tilde C_1\Delta t^2 + \tilde C_2 h^{2k} + \tilde C_3 h^{2l+2}, 
    \end{aligned}
$$
even in a limit situation where $a_0=b_0=c_0$.
Based on \textbf{Assumptions} (A1), the above inequality further implies that
$$
 \Delta t  \big(  
            k_m\Vert\nabla e_p^{h,N+1}\Vert_{L^2(\Omega)}^2 
        + \theta_m\Vert\nabla e_T^{h,N+1}\Vert_{L^2(\Omega)}^2\big)
    \le  \tilde C_1\Delta t^2 + \tilde C_2 h^{2k} + \tilde C_3 h^{2l+2},  
$$
which means that the optimal error estimates of $e_p^{h,N+1}$ and $e_T^{h,N+1}$ in the $H^1$ are obtained.
\end{remark}

\subsection{A priori estimate for the split sequential algorithms}
In this subsection, we present the a priori estimates for \textbf{Algorithm 1} and \textbf{Algorithm 2}.  
Since the results parallel those of the previous subsection and rely on the same proof techniques, we state the following theorems without providing detailed proofs.
\begin{theorem}[estimate for \textbf{Algorithm 1}]
    Let \((\bm u,~\xi,~p,~T)\) and \((\bm u_h^{n+1},~\xi_h^{n+1},~p_h^{n+1},~T_h^{n+1})\) for \(n\ge1\) be the solutions of Eqs. \eqref{TP_Model_var} and Eqs. \eqref{decoupled StR a}-\eqref{decoupled StR d}, respectively. 
    Let $e_{\bm u}^{n+1},~e_{\xi}^{n+1},~e_{p}^{n+1},~e_{T}^{n+1}$ be differences between the solutions at the time step $n+1$ of problem \eqref{TP_Model_var} and the numerical solutions of \textbf{Algorithm 1} and decompose the error terms into two parts like the form of \eqref{eq: e=eI+eh}.
    There holds
    \begin{equation}\label{}
    \begin{aligned} 
    & 
       \mu\Vert\bm\varepsilon(e_{\bm u}^{h,N+1})\Vert_{L^2(\Omega)}^2  
       +\Vert  e_\xi^{h,N+1}\Vert_{L^2(\Omega)}^2 
      +\tfrac{c_0-b_0}{2} \Vert  e_p^{h,N+1}\Vert_{L^2(\Omega)}^2  
       +\tfrac{a_0-b_0}{2} \Vert  e_T^{h,N+1}\Vert_{L^2(\Omega)}^2 
                      \\
                       &+  \Delta t \sum_{n=1}^N \big(  
        (\bm K\nabla e_p^{h,n+1},\nabla e_p^{h,n+1})  
        + (\bm\Theta\nabla e_T^{h,n+1},\nabla e_T^{h,n+1}) \big)\\
    \le &  \tilde C_1\Delta t^2 + \tilde C_2 h^{2k} + \tilde C_3 h^{2l+2},  
    \end{aligned}
    \end{equation} 
     where $\tilde C_1$, $\tilde C_2$, and $\tilde C_3$ are the same as  those of $\tilde C_1$, $\tilde C_2$, and $\tilde C_3$ in Theorem \ref{th: error estimate}.    
\end{theorem}
\begin{theorem}[estimate for \textbf{Algorithm 2}]
    Let \((\bm u,~\xi,~p,~T)\) and \((\bm u_h^{n+1},~\xi_h^{n+1},~p_h^{n+1},~T_h^{n+1})\) for \(n\ge1\) be the solutions of Eqs. \eqref{TP_Model_var} and Eqs. \eqref{decoupled RtS a}-\eqref{decoupled RtS d}, respectively. 
    Let $e_{\bm u}^{n+1},~e_{\xi}^{n+1},~e_{p}^{n+1},~e_{T}^{n+1}$ be differences between the solutions at the time step $n+1$ of problem \eqref{TP_Model_var} and the numerical solutions of \textbf{Algorithm 2} and decompose the error terms into two parts like the form of \eqref{eq: e=eI+eh}.
    There holds
    \begin{equation}\label{}
    \begin{aligned} 
    & 
       \mu\Vert\bm\varepsilon(e_{\bm u}^{h,N+1})\Vert_{L^2(\Omega)}^2  
       +\Vert  e_\xi^{h,N+1}\Vert_{L^2(\Omega)}^2 
      +\tfrac{c_0-b_0}{2} \Vert  e_p^{h,N+1}\Vert_{L^2(\Omega)}^2  
       +\tfrac{a_0-b_0}{2} \Vert  e_T^{h,N+1}\Vert_{L^2(\Omega)}^2 
                      \\
                       &+  \Delta t \sum_{n=1}^N \big(  
          (\bm K\nabla e_p^{h,n+1},\nabla e_p^{h,n+1})  
        + (\bm\Theta\nabla e_T^{h,n+1},\nabla e_T^{h,n+1}) \big)\\
    \le &  \tilde C_1\Delta t^2 + \tilde C_2 h^{2k} + \tilde C_3 h^{2l+2},  
    \end{aligned}
    \end{equation} 
    where $\tilde C_1$, $\tilde C_2$, and $\tilde C_3$ are the same as  those of $\tilde C_1$, $\tilde C_2$, and $\tilde C_3$ in Theorem \ref{th: error estimate}.    
\end{theorem}

\section{Numerical experiments}\label{sec: experiments}
In this section, we present numerical experiments to evaluate the computational accuracy and efficiency of the algorithms introduced in Section \ref{sec: algorithms}. 
The purpose is to examine the performance of the three schemes under various physical parameter settings.  
A uniform triangular mesh \(\mathcal{T}_h\) is generated over the domain \(\Omega = [0,1] \times [0,1]\), and the final time is set to \(\tau = 1\).  
The experiments are carried out on the following benchmark problems.

{\bfseries Example 1:}
We appropriately select the body force \(\bm f\), the mass source \(g\) and the heat source \(H_s\) to ensure that the exact solution is obtained.
\begin{equation} \label{Ex: example1}
\begin{aligned}   
      & \bm u(x,y,t)= e^{-t}\left(\begin{array}{c}  
     \sin(2\pi y)(\cos(2\pi x) -1)+\tfrac{1}{\mu+\lambda}\sin(\pi x) \sin(\pi y)  \\  
     \sin(2\pi x)(1-\cos(2\pi y) )+\tfrac{1}{\mu+\lambda}\sin(\pi x) \sin(\pi y)   \\
                                 \end{array}\right),   \\
      & p(x,y,t) =   e^{-t}\sin(\pi x) \sin(\pi y) ,      \\
      & T(x,y,t) =   e^{-t}\cos(\pi x) \cos(\pi y) .
\end{aligned}
\end{equation} 
We impose Dirichlet boundary conditions $\bm{u} = \bm{0}$ on the boundary segment $\Gamma_d = \{(x,y)\,|\, 0 \le y \le 1,\, x = 0 \text{ or } 1\}$, and Neumann boundary conditions on $\Gamma_n = \{(x,y)\,|\, 0 \le x \le 1,\, y = 0 \text{ or } 1\}$. For the pressure $p$ and temperature $T$, homogeneous Dirichlet boundary conditions are prescribed on the entire boundary $\partial\Omega$. The initial conditions for $\bm{u}^0$, $p^0$, and $T^0$ at time $t = 0$ are given by the expression in \eqref{Ex: example1}.

In the numerical experiments, we employ uniform triangular meshes with mesh sizes $h$ specified in the corresponding tables. These meshes are generated through successive midpoint bisections of each triangle. At the final simulation time $\tau$, we report the computed $L^2$-norm errors for $\xi$, and $H^1$-norm errors for $\bm u$, $p$, and $T$, along with their associated convergence rates.

Table \ref{tab: time order} reports the temporal convergence results.  
In this experiment, the mesh size is fixed at $h = 1/100$, with polynomial degrees $k = 3$ and $l = 2$ used for spatial discretization.  
The physical parameters are chosen as $\nu = 0.3$, $E = 1$, $c_0 = a_0 = 0.2$, $b_0 = 0.1$, $\alpha = \beta = 0.1$, and $\bm{K} = \bm{\Theta} = 0.1\bm{I}$, where $\bm{I}$ denotes the $2 \times 2$ identity matrix.  
As the time step $\Delta t$ is successively halved, the $H^1$-error of $\bm u$, the $L^2$-error of $\xi$, and the $H^1$-errors of $p$ and $T$ exhibit nearly first-order convergence, in agreement with the theoretical analysis for time error.

\begin{table}[htbp] \tiny  
\caption{ $\nu=0.3,E=1,a_0=c_0=0.2, b_0=0.1,\alpha=\beta=0.1,\bm K=\bm \Theta=\bm I,k=3,l=2 $. }
\centering
\begin{tabular}{cccccccccc}
\toprule
  $h$&$\Delta t$&$\Vert\bm u\Vert_{H^1}$ &rate &$\Vert\xi\Vert_{L^2}$ &rate &$\Vert p\Vert_{H^1}$ &rate&$\Vert T\Vert_{H^1}$ &rate \\ 
\midrule
Algorithm 1\\
\midrule
$1/100$ & $1/4$ & 6.85174e-03 & 0.00 & 4.24704e-03 & 0.00 & 2.61350e-03 & 0.00 & 2.47354e-03 & 0.00 \\
$1/100$ & $1/8$ & 3.20990e-03 & 1.09 & 1.98944e-03 & 1.09 & 1.26027e-03 & 1.05 & 1.18508e-03 & 1.06 \\
$1/100$ & $1/16$ & 1.55476e-03 & 1.05 & 9.63605e-04 & 1.05 & 6.21505e-04 & 1.02 & 5.83573e-04 & 1.02 \\
$1/100$ & $1/32$ & 7.65229e-04 & 1.02 & 4.74265e-04 & 1.02 & 3.15299e-04 & 0.98 & 2.96718e-04 & 0.98 \\
\midrule
Algorithm 2\\
\midrule
$1/100$ & $1/4$ & 2.59292e-04 & 0.00 & 1.70227e-04 & 0.00 & 7.60375e-03 & 0.00 & 7.58749e-03 & 0.00 \\
$1/100$ & $1/8$ & 1.15386e-04 & 1.17 & 7.56703e-05 & 1.17 & 3.39643e-03 & 1.16 & 3.38354e-03 & 1.17 \\
$1/100$ & $1/16$ & 5.50473e-05 & 1.07 & 3.59914e-05 & 1.07 & 1.61898e-03 & 1.07 & 1.61169e-03 & 1.07 \\
$1/100$ & $1/32$ & 2.71875e-05 & 1.02 & 1.75596e-05 & 1.04 & 7.93321e-04 & 1.03 & 7.89488e-04 & 1.03 \\
\midrule
Algorithm 3\\
\midrule
$1/100$ & $1/4$ & 6.89452e-03 & 0.00 & 4.27730e-03 & 0.00 & 7.49815e-03 & 0.00 & 7.50810e-03 & 0.00 \\
$1/100$ & $1/8$ & 3.20469e-03 & 1.11 & 1.98585e-03 & 1.11 & 3.38022e-03 & 1.15 & 3.39640e-03 & 1.14 \\
$1/100$ & $1/16$ & 1.55259e-03 & 1.05 & 9.62117e-04 & 1.05 & 1.61227e-03 & 1.07 & 1.61796e-03 & 1.07 \\
$1/100$ & $1/32$ & 7.64222e-04 & 1.02 & 4.73574e-04 & 1.02 & 7.90130e-04 & 1.03 & 7.92451e-04 & 1.03 \\
\bottomrule
\label{tab: time order}
\end{tabular}
\end{table}

\begin{table}[htbp] \tiny  
\caption{ $\nu=0.3,E=1,a_0=c_0=0.2, b_0=0.1,\alpha=\beta=0.1,\bm K=\bm \Theta=\bm I,k=2,l=1 $. }
\centering
\begin{tabular}{cccccccccc}
\toprule
  $h$&$\Delta t$&$\Vert\bm u\Vert_{H^1}$ &rate &$\Vert\xi\Vert_{L^2}$ &rate &$\Vert p\Vert_{H^1}$ &rate&$\Vert T\Vert_{H^1}$ &rate \\ 
\midrule
Algorithm 1\\
\midrule
$1/4$ & $1/4$ & 5.29575e-01 & 0.00 & 4.90883e-02 & 0.00 & 3.02299e-01 & 0.00 & 3.07582e-01 & 0.00 \\
$1/8$ & $1/16$ & 1.45378e-01 & 1.87 & 1.02454e-02 & 2.26 & 1.57993e-01 & 0.94 & 1.58713e-01 & 0.95 \\
$1/16$ & $1/64$ & 3.73916e-02 & 1.96 & 2.32919e-03 & 2.14 & 7.99174e-02 & 0.98 & 8.00095e-02 & 0.99 \\
$1/32$ & $1/256$ & 9.42485e-03 & 1.99 & 5.55704e-04 & 2.07 & 4.00760e-02 & 1.00 & 4.00876e-02 & 1.00 \\
\midrule
Algorithm 2\\
\midrule
$1/4$ & $1/4$ & 5.29752e-01 & 0.00 & 4.87496e-02 & 0.00 & 3.02362e-01 & 0.00 & 3.07643e-01 & 0.00 \\
$1/8$ & $1/16$ & 1.45398e-01 & 1.87 & 1.01787e-02 & 2.26 & 1.58000e-01 & 0.94 & 1.58720e-01 & 0.95 \\
$1/16$ & $1/64$ & 3.73924e-02 & 1.96 & 2.31723e-03 & 2.14 & 7.99182e-02 & 0.98 & 8.00103e-02 & 0.99 \\
$1/32$ & $1/256$ & 9.42468e-03 & 1.99 & 5.53153e-04 & 2.07 & 4.00761e-02 & 1.00 & 4.00877e-02 & 1.00 \\
\midrule
Algorithm 3\\
\midrule
$1/4$ & $1/4$ & 5.29575e-01 & 0.00 & 4.90926e-02 & 0.00 & 3.02360e-01 & 0.00 & 3.07640e-01 & 0.00 \\
$1/8$ & $1/16$ & 1.45378e-01 & 1.87 & 1.02454e-02 & 2.26 & 1.58000e-01 & 0.94 & 1.58720e-01 & 0.95 \\
$1/16$ & $1/64$ & 3.73916e-02 & 1.96 & 2.32917e-03 & 2.14 & 7.99182e-02 & 0.98 & 8.00103e-02 & 0.99 \\
$1/32$ & $1/256$ & 9.42485e-03 & 1.99 & 5.55697e-04 & 2.07 & 4.00761e-02 & 1.00 & 4.00877e-02 & 1.00 \\ 
\bottomrule
\label{tab: nu 0.3}
\end{tabular}
\end{table} 
We next examine the effect of the Poisson ratio $\nu$ on the performance of the algorithms, while keeping all other parameters fixed as described earlier.  
The results in Table \ref{tab: nu 0.3}, together with Tables \ref{tab: K Theta 10-9}--\ref{tab: k3l2}, correspond to the compressible case with $\nu = 0.3$.  
To probe the nearly incompressible regime, we set $\nu = 0.499$ in Table \ref{tab: nu 0.49}, with all remaining parameters unchanged.  
As $\nu \to 0.5$, implying $\lambda \to \infty$, the problem approaches the incompressible Stokes limit.  
Nevertheless, the numerical results show that all algorithms achieve optimal convergence in the energy norms, demonstrating robustness with respect to $\nu$.

\begin{table}[htbp] \tiny  
\caption{ $\nu=0.499,E=1,a_0=c_0=0.2, b_0=0.1,\alpha=\beta=0.1,\bm K=\bm \Theta=\bm I,k=2,l=1 $. }
\centering
\begin{tabular}{cccccccccc}
\toprule
  $h$&$\Delta t$&$\Vert\bm u\Vert_{H^1}$ &rate &$\Vert\xi\Vert_{L^2}$ &rate &$\Vert p\Vert_{H^1}$ &rate&$\Vert T\Vert_{H^1}$ &rate \\ 
\midrule
Algorithm 1\\
\midrule
$1/4$ & $1/4$ & 5.31525e-01 & 0.00 & 8.12402e-02 & 0.00 & 3.02292e-01 & 0.00 & 3.07576e-01 & 0.00 \\
$1/8$ & $1/16$ & 1.44897e-01 & 1.88 & 1.69559e-02 & 2.26 & 1.57992e-01 & 0.94 & 1.58713e-01 & 0.95 \\
$1/16$ & $1/64$ & 3.71835e-02 & 1.96 & 3.85633e-03 & 2.14 & 7.99173e-02 & 0.98 & 8.00094e-02 & 0.99 \\
$1/32$ & $1/256$ & 9.36343e-03 & 1.99 & 9.20079e-04 & 2.07 & 4.00760e-02 & 1.00 & 4.00876e-02 & 1.00 \\
\midrule
Algorithm 2\\
\midrule
$1/4$ & $1/4$ & 5.31526e-01 & 0.00 & 8.12391e-02 & 0.00 & 3.02292e-01 & 0.00 & 3.07576e-01 & 0.00 \\
$1/8$ & $1/16$ & 1.44897e-01 & 1.88 & 1.69557e-02 & 2.26 & 1.57992e-01 & 0.94 & 1.58713e-01 & 0.95 \\
$1/16$ & $1/64$ & 3.71835e-02 & 1.96 & 3.85631e-03 & 2.14 & 7.99173e-02 & 0.98 & 8.00094e-02 & 0.99 \\
$1/32$ & $1/256$ & 9.36343e-03 & 1.99 & 9.20078e-04 & 2.07 & 4.00760e-02 & 1.00 & 4.00876e-02 & 1.00 \\
\midrule
Algorithm 3\\
\midrule
$1/4$ & $1/4$ & 5.31525e-01 & 0.00 & 8.12402e-02 & 0.00 & 3.02292e-01 & 0.00 & 3.07576e-01 & 0.00 \\
$1/8$ & $1/16$ & 1.44897e-01 & 1.88 & 1.69559e-02 & 2.26 & 1.57992e-01 & 0.94 & 1.58713e-01 & 0.95 \\
$1/16$ & $1/64$ & 3.71835e-02 & 1.96 & 3.85633e-03 & 2.14 & 7.99173e-02 & 0.98 & 8.00094e-02 & 0.99 \\
$1/32$ & $1/256$ & 9.36343e-03 & 1.99 & 9.20079e-04 & 2.07 & 4.00760e-02 & 1.00 & 4.00876e-02 & 1.00 \\ 
\bottomrule
\label{tab: nu 0.49}
\end{tabular}
\end{table} 

\begin{table}[htbp] \tiny  
\caption{ $\nu=0.3,E=1,a_0=c_0=0.2, b_0=0.1,\alpha=\beta=0.1,\bm K=\bm \Theta=10^{-9}\bm I,k=2,l=1 $. }
\centering
\begin{tabular}{cccccccccc}
\toprule
  $h$&$\Delta t$&$\Vert\bm u\Vert_{H^1}$ &rate &$\Vert\xi\Vert_{L^2}$ &rate &$\Vert p\Vert_{H^1}$ &rate&$\Vert T\Vert_{H^1}$ &rate \\ 
\midrule
Algorithm 1\\
\midrule
$1/4$ & $1/4$ & 5.35740e-01 & 0.00 & 7.12752e-02 & 0.00 & 3.09616e+00 & 0.00 & 3.10761e+00 & 0.00 \\
$1/8$ & $1/16$ & 1.47489e-01 & 1.86 & 1.92937e-02 & 1.89 & 1.09761e+00 & 1.50 & 1.11094e+00 & 1.48 \\
$1/16$ & $1/64$ & 3.79680e-02 & 1.96 & 4.98155e-03 & 1.95 & 3.63574e-01 & 1.59 & 3.69518e-01 & 1.59 \\
$1/32$ & $1/256$ & 9.57336e-03 & 1.99 & 1.26006e-03 & 1.98 & 1.24491e-01 & 1.55 & 1.26756e-01 & 1.54 \\
\midrule
Algorithm 2\\
\midrule
$1/4$ & $1/4$ & 5.52266e-01 & 0.00 & 1.13085e-01 & 0.00 & 6.26198e+00 & 0.00 & 6.24313e+00 & 0.00 \\
$1/8$ & $1/16$ & 1.53288e-01 & 1.85 & 3.38974e-02 & 1.74 & 2.15557e+00 & 1.54 & 2.16011e+00 & 1.53 \\
$1/16$ & $1/64$ & 3.95654e-02 & 1.95 & 9.01412e-03 & 1.91 & 7.05097e-01 & 1.61 & 7.08109e-01 & 1.61 \\
$1/32$ & $1/256$ & 9.98569e-03 & 1.99 & 2.30565e-03 & 1.97 & 2.36564e-01 & 1.58 & 2.37853e-01 & 1.57 \\
\midrule
Algorithm 3\\
\midrule
$1/4$ & $1/4$ & 5.52089e-01 & 0.00 & 1.12960e-01 & 0.00 & 6.27555e+00 & 0.00 & 6.25513e+00 & 0.00 \\
$1/8$ & $1/16$ & 1.53521e-01 & 1.85 & 3.44360e-02 & 1.71 & 2.16177e+00 & 1.54 & 2.16556e+00 & 1.53 \\
$1/16$ & $1/64$ & 3.95747e-02 & 1.96 & 9.03953e-03 & 1.93 & 7.05553e-01 & 1.62 & 7.08252e-01 & 1.61 \\
$1/32$ & $1/256$ & 9.98508e-03 & 1.99 & 2.30525e-03 & 1.97 & 2.36612e-01 & 1.58 & 2.37777e-01 & 1.57 \\
\bottomrule
\label{tab: K Theta 10-9}
\end{tabular}
\end{table}

\begin{table}[htbp] \tiny  
\caption{ $\nu=0.3,E=1,a_0=c_0=b_0=0,\alpha=\beta=0.1,\bm K=\bm \Theta=\bm I,k=2,l=1 $. }
\centering
\begin{tabular}{cccccccccc}
\toprule
  $h$&$\Delta t$&$\Vert\bm u\Vert_{H^1}$ &rate &$\Vert\xi\Vert_{L^2}$ &rate &$\Vert p\Vert_{H^1}$ &rate&$\Vert T\Vert_{H^1}$ &rate \\ 
\midrule
Algorithm 1\\
\midrule
$1/4$ & $1/4$ & 5.29574e-01 & 0.00 & 4.90870e-02 & 0.00 & 3.02301e-01 & 0.00 & 3.07579e-01 & 0.00 \\
$1/8$ & $1/16$ & 1.45378e-01 & 1.87 & 1.02451e-02 & 2.26 & 1.57993e-01 & 0.94 & 1.58713e-01 & 0.95 \\
$1/16$ & $1/64$ & 3.73916e-02 & 1.96 & 2.32913e-03 & 2.14 & 7.99174e-02 & 0.98 & 8.00094e-02 & 0.99 \\
$1/32$ & $1/256$ & 9.42485e-03 & 1.99 & 5.55688e-04 & 2.07 & 4.00760e-02 & 1.00 & 4.00876e-02 & 1.00 \\
\midrule
Algorithm 2\\
\midrule
$1/4$ & $1/4$ & 5.29752e-01 & 0.00 & 4.87495e-02 & 0.00 & 3.02365e-01 & 0.00 & 3.07640e-01 & 0.00 \\
$1/8$ & $1/16$ & 1.45398e-01 & 1.87 & 1.01787e-02 & 2.26 & 1.58000e-01 & 0.94 & 1.58719e-01 & 0.95 \\
$1/16$ & $1/64$ & 3.73924e-02 & 1.96 & 2.31723e-03 & 2.14 & 7.99182e-02 & 0.98 & 8.00102e-02 & 0.99 \\
$1/32$ & $1/256$ & 9.42468e-03 & 1.99 & 5.53153e-04 & 2.07 & 4.00761e-02 & 1.00 & 4.00877e-02 & 1.00 \\
\midrule
Algorithm 3\\
\midrule
$1/4$ & $1/4$ & 5.29575e-01 & 0.00 & 4.90917e-02 & 0.00 & 3.02364e-01 & 0.00 & 3.07637e-01 & 0.00 \\
$1/8$ & $1/16$ & 1.45378e-01 & 1.87 & 1.02452e-02 & 2.26 & 1.58000e-01 & 0.94 & 1.58719e-01 & 0.95 \\
$1/16$ & $1/64$ & 3.73916e-02 & 1.96 & 2.32911e-03 & 2.14 & 7.99182e-02 & 0.98 & 8.00102e-02 & 0.99 \\
$1/32$ & $1/256$ & 9.42485e-03 & 1.99 & 5.55682e-04 & 2.07 & 4.00761e-02 & 1.00 & 4.00877e-02 & 1.00 \\
\bottomrule
\label{tab: a0=b0=c0}
\end{tabular}
\end{table} 

\begin{table}[htbp] \tiny  
\caption{ $\nu=0.3,E=1,a_0=c_0=0.2, b_0=0.1,\alpha=\beta=0.1,\bm K=\bm \Theta=\bm I,k=3,l=2 $. }
\centering
\begin{tabular}{cccccccccc}
\toprule
  $h$&$\Delta t$&$\Vert\bm u\Vert_{H^1}$ &rate &$\Vert\xi\Vert_{L^2}$ &rate &$\Vert p\Vert_{H^1}$ &rate&$\Vert T\Vert_{H^1}$ &rate \\ 
\midrule
Algorithm 1\\
\midrule
$1/4$ & $1/4$ & 8.05686e-02 & 0.00 & 6.45457e-03 & 0.00 & 4.57150e-02 & 0.00 & 4.70966e-02 & 0.00 \\
$1/8$ & $1/32$ & 9.85829e-03 & 3.03 & 7.87602e-04 & 3.03 & 1.20176e-02 & 1.93 & 1.21952e-02 & 1.95 \\
$1/16$ & $1/256$ & 1.20740e-03 & 3.03 & 9.80993e-05 & 3.01 & 3.06313e-03 & 1.97 & 3.08539e-03 & 1.98 \\
$1/32$ & $1/2048$ & 1.49569e-04 & 3.01 & 1.23106e-05 & 2.99 & 7.71746e-04 & 1.99 & 7.74526e-04 & 1.99 \\
\midrule
Algorithm 2\\
\midrule
$1/4$ & $1/4$ & 8.05772e-02 & 0.00 & 4.82056e-03 & 0.00 & 4.62655e-02 & 0.00 & 4.76362e-02 & 0.00 \\
$1/8$ & $1/32$ & 9.83672e-03 & 3.03 & 6.31639e-04 & 2.93 & 1.20397e-02 & 1.94 & 1.22171e-02 & 1.96 \\
$1/16$ & $1/256$ & 1.20392e-03 & 3.03 & 7.91288e-05 & 3.00 & 3.06442e-03 & 1.97 & 3.08668e-03 & 1.98 \\
$1/32$ & $1/2048$ & 1.49111e-04 & 3.01 & 9.93984e-06 & 2.99 & 7.71826e-04 & 1.99 & 7.74606e-04 & 1.99 \\
\midrule
Algorithm 3\\
\midrule
$1/4$ & $1/4$ & 8.05720e-02 & 0.00 & 6.47332e-03 & 0.00 & 4.62485e-02 & 0.00 & 4.76236e-02 & 0.00 \\
$1/8$ & $1/32$ & 9.85828e-03 & 3.03 & 7.87035e-04 & 3.04 & 1.20395e-02 & 1.94 & 1.22173e-02 & 1.96 \\
$1/16$ & $1/256$ & 1.20740e-03 & 3.03 & 9.80409e-05 & 3.00 & 3.06440e-03 & 1.97 & 3.08669e-03 & 1.98 \\
$1/32$ & $1/2048$ & 1.49568e-04 & 3.01 & 1.23040e-05 & 2.99 & 7.71825e-04 & 1.99 & 7.74607e-04 & 1.99 \\
\bottomrule
\label{tab: k3l2}
\end{tabular}
\end{table}

\begin{table}[htbp] \tiny  
\caption{ $\nu=0.3,E=1,a_0=c_0=0.2, b_0=0.1,\alpha=\beta=0.1,\bm K=\bm \Theta=\bm I,k=2,l=1 $. }
\centering
\begin{tabular}{cccccccccc}
\toprule
  &$h$&$\Delta t$&$\Vert\bm u\Vert_{H^1}$  &$\Vert\xi\Vert_{L^2}$  &$\Vert p\Vert_{H^1}$ &$\Vert T\Vert_{H^1}$ & CPU time(s) \\ 
\midrule
Fully  Coupled Algorithm  &$1/40$&$1/16$&1.24321e-02&7.27150e-04&4.79393e-02&4.79393e-02&18.81\\
Algorithm 1   &$1/40$&$1/16$&6.39521e-03&1.51777e-03&3.20763e-02&3.20763e-02&17.72\\
Algorithm 2  &$1/40$&$1/16$&6.03897e-03&3.54112e-04&3.21111e-02&3.21111e-02&17.62\\
Algorithm 3(parallel algorithm) &$1/40$&$1/16$&6.35586e-03&1.40841e-03&3.21112e-02&3.21112e-02&12.99\\
\midrule
Fully  Coupled algorithm  &$1/80$&$1/64$&3.07670e-03&1.78298e-04&1.82882e-02&1.82882e-02&315.37\\
Algorithm 1   &$1/80$&$1/64$&1.59730e-03&3.71409e-04&1.60441e-02&1.60441e-02&218.92\\
Algorithm 2  &$1/80$&$1/64$&1.51243e-03&8.69072e-05&1.60481e-02&1.60481e-02&218.34\\
Algorithm 3(parallel algorithm) &$1/80$&$1/64$&1.89545e-03&7.52876e-04&1.60482e-02&1.60482e-02&149.74\\
\bottomrule
\label{tab: time compare}
\end{tabular}
\end{table} 
Next, we investigate the performance of the algorithms under challenging parameter regimes characterized by extremely low hydraulic and thermal conductivities.  
In Table \ref{tab: nu 0.3}, the baseline case was considered with $\bm{K} = \bm{\Theta} = 0.1\bm{I}$.  
To stress-test the algorithms, Table \ref{tab: K Theta 10-9} reports results for the extreme case $\bm{K} = \bm{\Theta} = 10^{-9}\bm{I}$, while keeping all other parameters unchanged.  
This setting corresponds to a nearly impermeable and thermally insulating medium, where diffusion effects are negligible.  
Remarkably, even under such restrictive conditions, all algorithms preserve their optimal convergence behavior, thereby confirming their reliability and robustness in near-degenerate conductivity regimes.  

We also examine the sensitivity of the methods to material parameters associated with thermal and storage effects.  
Specifically, Table \ref{tab: a0=b0=c0} presents results for the degenerate configuration $a_0 = b_0 = c_0 = 0$, with the remaining parameters identical to those in Table \ref{tab: nu 0.3}.  
This test eliminates the effective thermal parameter, the thermal dilation coefficient, and the specific storage coefficient, thereby probing the limits of the thermo-poroelastic model.  
Despite the absence of these stabilizing contributions, all algorithms continue to yield optimal convergence rates in the energy norms, which highlights their robustness even in highly singular parameter settings.  

Taken together, these numerical experiments demonstrate that the proposed algorithms are unconditionally stable, achieve optimal convergence, and exhibit strong robustness across a broad spectrum of physical parameter regimes.  
An additional practical advantage lies in their ability to decouple the fully coupled four-field problem into two smaller subproblems.  
In particular, the parallel semi-decoupled scheme (\textbf{Algorithm 3}) is especially attractive, as it not only maintains accuracy but also offers substantial reductions in computational cost.  

To substantiate this point, Table \ref{tab: time compare} compares the CPU times of the fully coupled method \eqref{TP_Model_dis_a}--\eqref{TP_Model_dis_d} with those of \textbf{Algorithms 1}, \textbf{2}, and \textbf{3}.  
The results clearly show that all three algorithms significantly reduce computational effort without sacrificing accuracy, with the parallel implementation (\textbf{Algorithm 3}) delivering the most notable efficiency gains.

{\bfseries Example 2:}  
In this example, we consider a geothermal reservoir problem involving injection–production processes in a fractured domain, inspired by \cite{yi2024physics, gudala2024fractured}. The computational domain is given by $\Omega = (0,500)\times(0,500) $, containing a $200$ long hydraulic fracture that serves as the primary pathway for fluid flow and heat exchange with the surrounding rock, as illustrated in Fig.~\ref{fig: domain}. The injection and production wells are placed at $(350,250)$ and $(150,250)$, respectively, at the two ends of the fracture.  

To simplify the setup, the injection and production functions for pressure and temperature are prescribed as
\begin{equation*}
\begin{aligned}
&g(x,y) = 100 e^{(-0.001(x - 150)^2 - 0.001(y - 250)^2)} - 100 e^{(-0.001(x - 350)^2 - 0.001(y - 250)^2)}, \\
&H_s(x,y) = 100 e^{(-0.001(x - 150)^2 - 0.001(y - 250)^2)} - 100 e^{(-0.001(x - 350)^2 - 0.001(y - 250)^2)},
\end{aligned}
\end{equation*}
with body force $\bm f = \bm 0$. The discretization parameters are $h=10$, $\Delta t = 0.01$, and final time $\tau = 1$.  

We choose $\nu = 0.499$ (close to the incompressible limit $\nu \to 0.5$) and $E=24$, leading to a large value of $\lambda$. The thermal dilation coefficient is set to $b_0=3 \times 10^{-5}$, and the fluid viscosity is taken to be low, $\mu_f=3.2 \times 10^{-10}$. With fluid permeability $\bm K_p = 3.2 \times 10^{-16}\bm I$, the effective matrix parameter becomes $\bm K = \bm K_p / \mu_f = 10^{-6}\bm I$. The remaining parameters are given by
\begin{equation*}
\begin{aligned}
a_0 = 0.1, \quad c_0 = 10^{-3}, \quad \alpha = 0.25, \quad \beta = 0.001, \quad \bm\Theta = 2.6 \bm I .
\end{aligned}
\end{equation*}
Boundary conditions are imposed as in \eqref{eq: B_C}, with initial states $\bm u^0 = \bm 0$, $p^0 = 0$, and $T^0 = 100$.  

Since the three algorithms yield nearly identical results, we directly present the displacement $\bm u$, pressure $p$, and temperature $T$ obtained by \textbf{Algorithm 3} in Fig.~\ref{fig: pTu}. The computed fields capture the essential physical processes of fluid injection, heat transport, and displacement in the fractured geothermal reservoir. Importantly, these results demonstrate that \textbf{Algorithm 3} not only maintains accuracy comparable to the sequential methods but also achieves this with substantially improved computational efficiency. This highlights its effectiveness and practicality for large-scale thermo-poroelastic simulations. 

 \begin{figure}
     \centering 
     \includegraphics[width=0.3\textwidth]{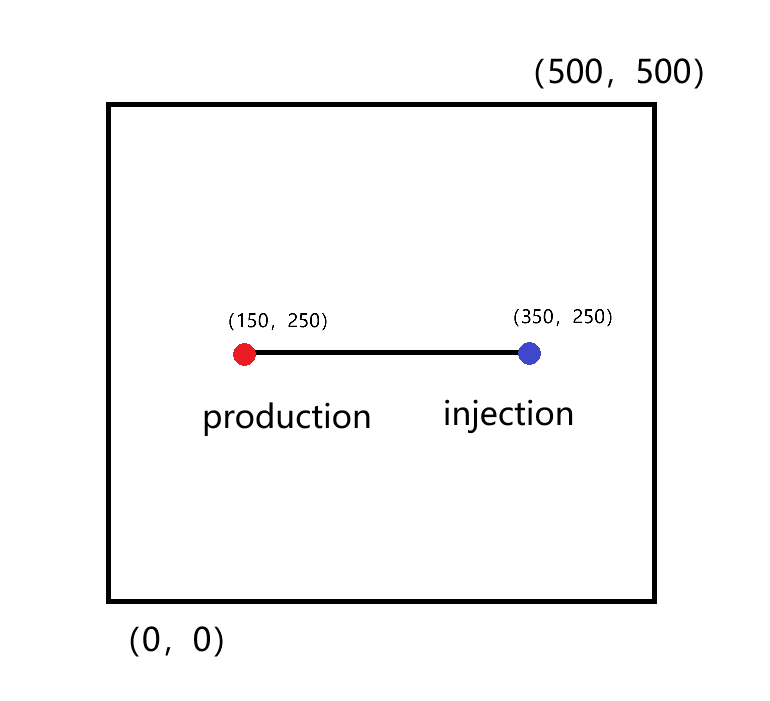
     }
     \vspace{-15pt}
     \caption{Injection-production domain
     }
     \label{fig: domain}
 \end{figure}
 
 \begin{figure}
     \centering 
     \includegraphics[width=0.3\textwidth]{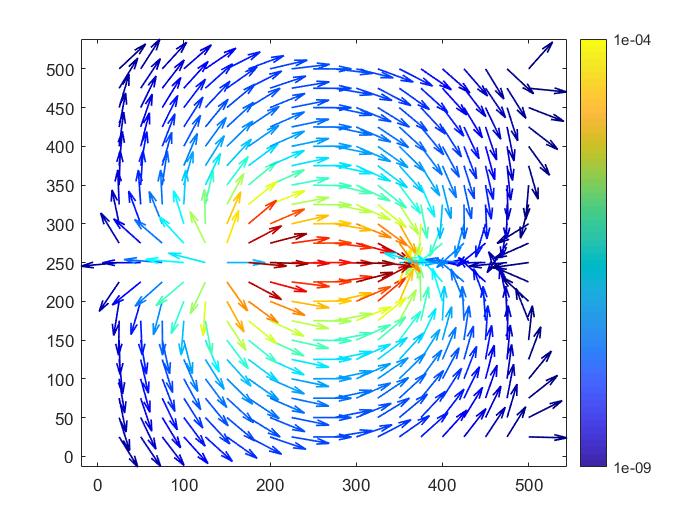
     }
     \includegraphics[width=0.3\textwidth]{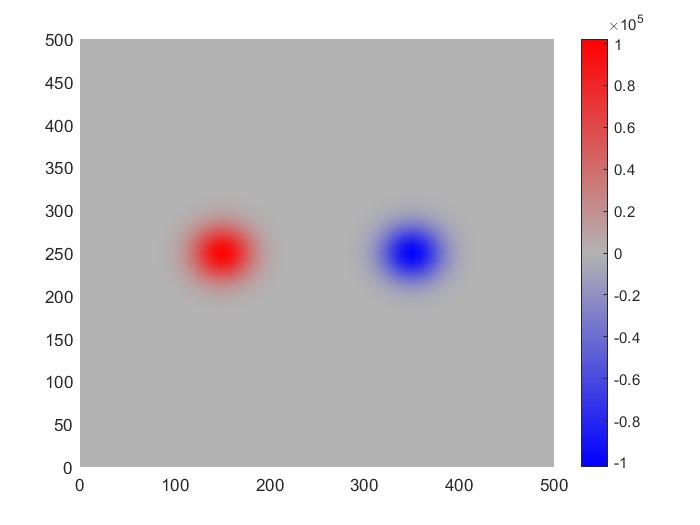
     }
     \includegraphics[width=0.3\textwidth]{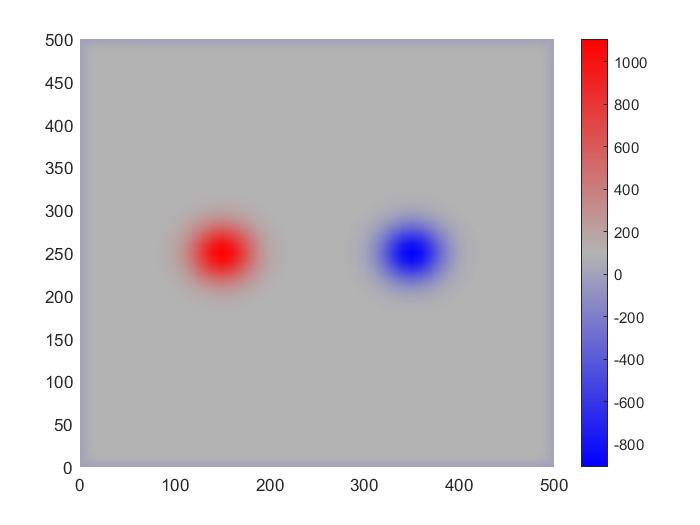
     }
     \vspace{-15pt}
     \caption{Example 2: displacement $\bm u$(left), pressure $p$(middle), and temperature $T$ at terminal time $\tau=1$
     }
     \label{fig: pTu}
 \end{figure}

\section{Conclusions}\label{sec: conclusions}
In this work, we developed and analyzed three non-iterative semi-decoupled algorithms for the thermo-poroelastic model: two sequentially split schemes and one parallel split scheme. Numerical experiments are provided to validate the effectiveness and the efficiency of these algorithms. The parallel algorithm leverages a two-step time discretization to achieve true temporal decoupling, enabling independent and concurrent solution of the two subsystems at each time level. Compared to existing methods, our approach avoids the need for stabilization terms while maintaining stability and accuracy across a wide range of parameter regimes.  These features make the proposed algorithms highly efficient, robust, and accurate for large-scale multiphysics simulations in geomechanics, biomedical modeling, and energy applications.

\section*{Acknowledgments}
The authors would like to thank Dr. Jijing Zhao for sharing her idea on a parallel algorithm for multiple-network Biot's model \cite{zhao2025optimally}, which motivates this work. The work of J. Li is partially supported by the Shenzhen Sci-Tech Fund No. RCJC20200714114556020, Guangdong Basic and Applied Research Fund No. 2023B1515250005.  The work of M. Cai is partially supported by the NIH-RCMI award (Grant No. 347U54MD013376) and the affiliated project award from the Center for Equitable Artificial Intelligence and Machine Learning Systems (CEAMLS) at Morgan State University (project ID 02232301). The work of Q. Liu is partially supported by the Guangdong Basic and Applied Basic Research Foundation (2022B1515120009), and the Shenzhen Science and Technology Program  (20231121110406001, JCYJ20241202124209011).

\bibliographystyle{abbrv}

\bibliography{ref}

\end{document}